\documentclass[a4paper,12pt,reqno]{amsart}
\usepackage{latexsym}
\usepackage{amssymb} 
\usepackage{mathrsfs}
\usepackage{amsmath}
\usepackage{latexsym}
\usepackage{delarray}
\usepackage{amssymb,amsmath,amsfonts,amsthm,mathrsfs}
\usepackage{blkarray}

\usepackage{hyperref}
\renewcommand\eqref[1]{(\ref{#1})} 
 \newtheorem{thm}{Theorem}[section]
 \newtheorem{cor}[thm]{Corollary}
 \newtheorem{lem}[thm]{Lemma}
 \newtheorem{prop}[thm]{Proposition}
 \theoremstyle{definition}
 
 \theoremstyle{remark}
 \newtheorem{rem}[thm]{Remark}
 
 \numberwithin{equation}{section}
\newcommand{\half}{\frac{1}{2}}

\newcommand{\ene}{\mathbb{N}}

\newcommand{\er}{\mathbb{R}}
\newcommand{\ce}{\mathbb{C}}

\newcommand{\tn}{\mathbb{T}^n}
\newcommand{\T}{\mathbb{T}^1}

\newcommand{\efee}{\mathcal{F}}

\newcommand{\efeM}{{\mathcal{F}_M}}
\newcommand{\efeG}{{\mathcal{F}_G}}
\newcommand{\efemo}{\mathcal{F}_M\otimes\overline{\mathcal{F}_M}}

\newcommand{\bi}{\begin{itemize}}
\newcommand{\cinfm}{{C}^{\infty}(M)}

\newcommand{\ei}{\end{itemize}}
\newcommand{\be}{\begin{enumerate}}
\newcommand{\ee}{\end{enumerate}}
\newcommand{\beq}{\begin{equation}}
\newcommand{\eq}{\end{equation}}

\newcommand{\cdxi}{\ce^{d_{\xi}\times d_{\xi}}}

\newcommand{\Dcal}{\mathcal{D}}

\def\p#1{{\left({#1}\right)}}

\def\jp#1{{\left\langle{#1}\right\rangle}}

\def\Op{{{\rm Op}}}

\def\En{{E_{o}}}

\def\Hcal{{\mathcal H}}

\DeclareMathOperator{\Tr}{Tr}

\def\Gh{{\widehat{G}}}

\def\HS{{\mathtt{HS}}}
\def\Rn{{{\mathbb R}^n}}
\def\Tn{{{\mathbb T}^n}}
\def\Zn{{{\mathbb Z}^n}}
\def\T{{{\mathbb T}^1}}
\def\N{{{\mathbb N}}}
\def\C{{{\mathbb C}}}
\def\SU2{{{\rm SU(2)}}}
\def\lapsu2{{{\mathcal L}_\SU2}}

\def\Op{\text{\rm Op}}

\begin{document}

%
%
%
%
%
%
%
%
%
\title[Fourier multipliers, symbols and nuclearity on compact manifolds]
 {Fourier multipliers, symbols and nuclearity on compact manifolds}

\author[Julio Delgado]{Julio Delgado}

\address{%
Department of Mathematics\\
Imperial College London\\
180 Queen's Gate, London SW7 2AZ\\
United Kingdom
}
\email{j.delgado@imperial.ac.uk}

\thanks{The first author was supported by Marie Curie IIF 301599 and by
the Leverhulme Grant RPG-2014-02. 
The second author was supported by EPSRC grant EP/K039407/1. 
No new data was collected or generated during the course of the research.}
\author[Michael Ruzhansky]{Michael Ruzhansky}

\address{%
Department of Mathematics\\
Imperial College London\\
180 Queen's Gate, London SW7 2AZ\\
United Kingdom
}

\email{m.ruzhansky@imperial.ac.uk}



\date{\today}
\begin{abstract}
The notion of invariant operators, or Fourier multipliers, is discussed for densely
defined operators on Hilbert spaces, with respect to a fixed partition of the space
into a direct sum of finite dimensional subspaces. As a consequence,
given a compact manifold $M$ endowed with a positive measure, we introduce a notion 
of the operator's full symbol adapted to the Fourier analysis relative to a fixed 
elliptic operator $E$. 
We give a description of Fourier multipliers, or of operators invariant relative to $E$.
We apply these concepts to study Schatten classes of operators on $L^2(M)$
and to obtain a formula for the trace of trace class operators.
We also apply it to provide conditions for operators between $L^{p}$-spaces
to be $r$-nuclear in the sense of Grothendieck. 
\end{abstract}

\maketitle

\section{Introduction}
Let $M$ be a closed manifold (i.e. a  compact smooth manifold without boundary) of 
dimension $n$ endowed with a positive measure $dx$. 
Given an elliptic positive pseudo-differential operator $E$ of order $\nu$ on $M$, 
by considering an orthonormal basis consisting of eigenfunctions of $E$ we will 
associate a discrete Fourier analysis to the operator $E$ in the sense introduced 
by Seeley (\cite{see:ex}, \cite{see:exp}). 
This analysis allows us to introduce further a notion of invariant operators and of 
matrix-symbols corresponding to those operators. 
The operators on $M$ will be then analysed in terms of the corresponding symbols 
relative to the operator $E$. 

As a general framework, we first discuss invariant operators, or Fourier 
multipliers in a general Hilbert space $\Hcal$. This notion is based on a 
partition of $\Hcal$ into a direct sum of finite dimensional subspaces, so that
a densely defined operator on $\Hcal$ can be decomposed as acting in these
subspaces. There are two main examples of this construction discussed in the 
paper: operators on $\Hcal=L^{2}(M)$ for a compact manifold $M$ as well as 
operators on $\Hcal=L^{2}(G)$ for a compact Lie group $G$. The difference in 
approaches to these settings is in the choice of partitions of $\Hcal$ into direct sums of
subspaces: in the former case they are chosen as eigenspaces of a fixed
elliptic pseudo-differential operator 
on $M$ while in the latter case they are chosen as linear spans 
of matrix coefficients of inequivalent irreducible unitary representations of $G$.

We note that for some results, the self-adjointness and ellipticity of $E$ can be dropped,
see \cite{Ruzhansky-Tokmagambetov:IMRN}.

We give applications of these notions to the derivation of conditions 
characterising those invariant operators on
$L^{2}(M)$ that belong to Schatten classes.
Furthermore, we also give conditions for nuclearity on $L^{p}$-spaces and,
more generally, for the $r$-nuclearity of operators. While the theory 
of $r$-nuclear operators in general Banach spaces has been developed by
Grothendieck \cite{gro:me} with numerous further advances
(e.g. in \cite{hi:pn,ko:pn,Oloff:pnorm,piet:r,Reinov}),
in this paper we give conditions in terms of
symbols for operators to be $r$-nuclear from $L^{p_{1}}(M)$ to $L^{p_{2}}(M)$
for $1\leq p_{1},p_{2}<\infty$ and $0<r\leq 1$. Consequently, we determine
relations between $p_{1}, p_{2}, r$ and $\alpha$ ensuring that the powers
$(I+E)^{-\alpha}$ are $r$-nuclear.
Trace formulas are also obtained relating operator traces to expressions involving
their symbols. 

In the recent work \cite{dr:suffkernel} the authors found sufficient conditions for 
operators to belong to Schatten classes $S_{p}$ on compact manifolds in terms of 
their Schwartz integral kernels. For $p<2$, it is customary to impose regularity
conditions on the kernel because there are counterexamples 
to conditions formulated only in terms of the integrability of kernels. Such examples
go back to Carleman's work \cite{car:ex} and their relevance to Schatten classes
has been discussed in \cite{dr13a:nuclp}.
A characteristic feature of conditions of this paper is that no regularity is assumed
neither on the symbol nor on the kernel. 
In the case of compact Lie groups, our results extend results on
Schatten classes and on $r$-nuclear operators on $L^p$ spaces that have been
obtained in  \cite{dr13:schatten} and \cite{dr13a:nuclp}. We show this by relating
the symbols introduced in this paper to matrix-valued symbols on compact
Lie groups developed in 
\cite{rt:groups} and in \cite{rt:book}.

Schatten classes of pseudo-differential operators in the setting of the 
Weyl-H\"or\-man\-der calculus have been considered in 
\cite{Toft:Schatten-AGAG-2006}, \cite{Toft:Schatten-modulation-2008}, 
\cite{Buzano-Nicola:2004},
\cite{Buzano-Nicola:powers-hypo-JFA-2007},
\cite{Buzano-Toft:Schatten-Weyl-JFA-2010}. Conditions for symbols of
lower regularity we given in \cite{sob:sch}. For the global analysis of
pseudo-differential operators on $\Rn$ see \cite{Boggiatto-Buzano-Rodino:bk},
as well as \cite[Chapter 4]{Nicola-Rodino:bk} also for the basic general introduction to
Schatten classes.

To formulate the notions more precisely, let $\Hcal$ be a complex Hilbert space and let 
$T:\Hcal\rightarrow \Hcal$ be a linear compact operator. If we denote by 
$T^*:\Hcal\rightarrow \Hcal$ the adjoint of  $T$, then the linear operator 
$(T^*T)^\half:\Hcal\rightarrow \Hcal$ is positive and compact. 
Let $(\psi_k)_k$ be an orthonormal basis for $\Hcal$ consisting of eigenvectors of 
$|T|=(T^*T)^\half$, and let $s_k(T)$ be the eigenvalue corresponding to the eigenvector 
$\psi_k$, $k=1,2,\dots$. The non-negative numbers $s_k(T)$, $k=1,2,\dots$, 
are called the singular values of $T:\Hcal\rightarrow \Hcal$. 
If $0<p<\infty$ and the sequence of singular values is $p$-summable, then $T$ 
is said to belong to the Schatten class  ${S}_p(\Hcal)$, and it is well known that each 
${S}_p(\Hcal)$ is an ideal in $\mathscr{L}(\Hcal)$. 
If $1\leq p <\infty$, a norm is associated to ${S}_p(\Hcal)$ by
 \[
 \|T\|_{S_p}=\left(\sum\limits_{k=1}^{\infty}(s_k(T))^p\right)^{\frac{1}{p}}.
 \] 
If $1\leq p<\infty$ 
 the class $S_p(\Hcal)$ becomes a Banach space endowed by the norm $\|T\|_{S_p}$. 
 If $p=\infty$ we define $S_{\infty}(\Hcal)$ as the class of bounded linear operators on $\Hcal$, 
 with 
$\|T\|_{S_\infty}:=\|T\|_{op}$, the operator norm.  For the Schatten class $S_2$ we will sometimes write $\|T\|_{\HS}$ instead of $\|T\|_{S_2}$. In the case $0<p<1$  the quantity $\|T\|_{S_p}$ only defines a  quasi-norm, and $S_p(\Hcal)$ is also complete. 
The space $S_1(\Hcal)$ is known as the {\em trace class} and an element of
 $S_2(\Hcal)$ is usually said to be a {\em Hilbert-Schmidt} operator. 
 For the  basic theory of Schatten classes we refer the reader to 
 \cite{gokr}, \cite{r-s:mp}, \cite{sim:trace}, \cite{sch:id}. 

It is well known that the class $S_2(L^2)$ is characterised by the   square integrability of the corresponding integral kernels, however, kernel estimates of this type are not effective
for classes $S_p(L^2)$ with $p<2$. This is explained by a classical Carleman's example
\cite{car:ex} on the summability of Fourier coefficients of continuous functions 
(see \cite{dr13a:nuclp} for a complete explanation of this fact). 
This obstruction explains the relevance of symbolic Schatten criteria and here we will clarify 
the advantage of the symbol approach with respect to this obstruction.
With this approach, no regularity of the kernel needs to be assumed.

In Section \ref{SEC:Lie-groups} we discuss the relation of our approach to that of
the global analysis on compact Lie groups.
In particular, in the case of compact
Lie groups the Fourier coefficients can be arranged into a (square)
matrix rather than in a column
leading to several simplifications. On general compact manifolds, 
this is not possible since the multiplicities
$d_{j}$ do not need to be all squares of integers.

We introduce $\ell^{p}$-style norms on the space of symbols $\Sigma$,
yielding discrete spaces $\ell^{p}(\Sigma)$ for $0<p\leq\infty$, normed
for $p\geq 1$.
Denoting by $\sigma_{T}$ the matrix symbol of an invariant operator $T$ provided by
Theorem \ref{THM:inv}, Schatten classes of invariant operators 
on $L^{2}(M)$ can be characterised
concisely by conditions
\begin{equation}\label{EQ:thm-L2-1}
T\in {\mathscr L}(L^{2}(M)) \Longleftrightarrow \sigma_{T}\in \ell^{\infty}(\Sigma),
\end{equation}
and for $0<p<\infty$,
\begin{equation}\label{EQ:thm-Sp-1}
T\in S_{p}(L^{2}(M)) \Longleftrightarrow \sigma_{T}\in \ell^{p}(\Sigma),
\end{equation}
see \eqref{EQ:thm-L2} and \eqref{EQ:thm-Sp}. 
Here, the condition that $T$ is invariant will mean that $T$ is strongly 
commuting with $E$ (see Theorem \ref{THM:inv}). 
On the level of the Fourier transform this means that
$$\widehat{Tf}(\ell)=\sigma(\ell) \widehat{f}(\ell)$$ for a family of
matrices $\sigma(\ell)$, i.e. $T$ assumes the familiar form of a 
Fourier multiplier.

In Section \ref{SEC:abstract} in 
Theorem \ref{THM:inv-rem} we discuss the abstract notion of symbol for
operators densely defined in a general Hilbert space $\Hcal$, and give several alternative
formulations for invariant operators, or for Fourier multipliers, relative to a fixed
partition of $\Hcal$ into a direct sum of finite dimensional subspaces,
$$\Hcal=\bigoplus_{j} H_{j}.$$
Consequently, in Theorem \ref{L2-abstract} 
we give the necessary and sufficient condition for the bounded extendability
of an invariant operator to ${\mathscr L}(\Hcal)$ in terms of its symbol, and in 
Theorem \ref{schchr-abstract} the necessary and sufficient condition for the
operator to be in Schatten classes $S_{r}(\Hcal)$ for $0<r<\infty$, as well
as the trace formula for operators in the trace class $S_{1}(\Hcal)$ in terms
of their symbols. 
As our subsequent analysis relies to a large extent on properties of
elliptic pseudo-differential operators on $M$, in Sections
\ref{SEC:Fourier} and \ref{SEC:invariant} we specify this abstract analysis
to the setting of operators densely defined on $L^{2}(M)$. The main difference
is that we now adopt the Fourier analysis to a fixed elliptic positive
pseudo-differential operator $E$ on $M$, contrary to the case of 
an operator $\En\in {\mathscr L}(\Hcal)$ in Theorem \ref{THM:inv-rem2}.

The notion of invariance depends on the choice of the spaces $H_{j}$.
Thus, in the analysis of operators on $M$ we take $H_{j}$'s to be 
the eigenspaces of $E$. However, other choices are possible. For example,
for $\Hcal=L^{2}(G)$ for a compact Lie group $G$, choosing $H_{j}$'s as
linear spans of representation coefficients for inequivalent irreducible
unitary representations of $G$, we make a link to the
quantization of pseudo-differential operator on compact Lie groups
as in \cite{rt:book}. These two partitions coincide when inequivalent representations of
$G$ produce distinct eigenvalues of the Laplacian; for example, this is the case
for $G={\rm SO(3)}.$ However, the partitions are different when inequivalent
representations produce equal eigenvalues, which is the case, for example, for 
$G={\rm SO(4)}.$ For the more explicit example on $\Hcal=L^{2}(\Tn)$ on the
torus see Remark \ref{REM:torus}. A similar choice could be made in other
settings producing a discrete spectrum and finite dimensional eigenspaces,
for example for operators in Shubin classes on $\Rn$, see Chodosh 
\cite{Chodosh} for the case $n=1$.

The analogous concept to Schatten classes in the setting of Banach spaces is
the notion of $r$-nuclearity introduced by Grothendieck \cite{gro:me}.
It has applications to questions of the distribution of eigenvalues of operators
in Banach spaces. 
In the setting of compact Lie groups these applications have been 
discussed in \cite{dr13a:nuclp} and they include conclusions on the 
distribution or summability of eigenvalues of operators 
acting on $L^{p}$-spaces. Another application is the 
Grothendieck-Lidskii formula which is the formula for the trace of operators
on $L^{p}(M)$.
Once we have $r$-nuclearity, most of further arguments are then purely 
functional analytic, so they apply equally well
in the present setting of closed manifolds. Because of this we omit the 
repetition of statements and refer the reader
to \cite{dr13a:nuclp} for further such applications.

Some results of this paper have been announced in \cite{Delgado-Ruzhansky:CRAS-kernels}, so here we provide their proofs, including a correction to the formulation of
\cite[Theorem 3.1, (iv)]{Delgado-Ruzhansky:CRAS-kernels}
given by Theorem \ref{THM:inv}, (iv), of this paper.

The paper is organised as follows. 
In Section \ref{SEC:abstract} we discuss Fourier multipliers and their symbols in
general Hilbert spaces.
In Section \ref{SEC:Fourier} we associate a
global Fourier analysis to an elliptic positive pseudo-differential operator $E$ on 
a closed manifold $M$. In Section \ref{SEC:invariant} we introduce the class
of operators invariant relative to $E$ as well as their matrix-valued symbols,
and apply this to characterise invariant operators in Schatten classes in
Section \ref{SEC:Schatten-mfds}.
In Section \ref{SEC:Lie-groups} we relate the analysis developed
so far to the analysis on compact Lie groups from
\cite{rt:groups}, \cite{rt:book}, and establish formula relating their matrix symbols
in the case when $M$ is a compact Lie group.
In particular, we will see that left-invariant operators on compact Lie groups
are invariant in our sense.
In Section \ref{SEC:kernels} we analyse the integral kernels of invariant
operators on general closed manifolds. 
Finally, in Section \ref{SEC:nuclearity} we apply our analysis to study $r$-nuclear
operators on $L^{p}$-spaces. 

Throughout the paper, we denote $\N_{0}=\N\cup\{0\}$.
Also $\delta_{j\ell}$ will denote the Kronecker delta, i.e.
$\delta_{j\ell}=1$ for $j=\ell$, and $\delta_{j\ell}=0$ for $j\not=\ell$.

The authors would like to thank V\'eronique Fischer, Alexandre Kirilov, and Augusto Almeida de Moraes Wagner for comments.

\section{Fourier multipliers in Hilbert spaces}
\label{SEC:abstract}

In this section we present an abstract set up to describe what we will
call invariant operators, or Fourier multipliers, acting on a general
Hilbert space $\Hcal$. We will give several characterisations of such
operators and their symbols. Consequently, we will apply these notions
to describe several properties of the operators, in particular, their
boundedness on $\Hcal$ as well as the Schatten properties.

We note that direct integrals (sums in our case) of Hilbert spaces
have been investigated in a much greater generality, see e.g.
Bruhat \cite{Bruhat:BK-Tata-1968},
Dixmier \cite[Ch 2., \S 2]{Dixmier:bk-algebras},
\cite[Appendix]{Dixmier:bk-Cstar-algebras}. The setting required for our
analysis is much simpler, so we prefer to adapt it specifically for
consequent applications, also providing short proofs for our statements.

The main application of the constructions below will be in the setting when
$M$ is a compact manifold without boundary, $\Hcal=L^{2}(M)$ and
$\Hcal^{\infty}=C^{\infty}(M)$, which will be described in detail
in Section \ref{SEC:Fourier}. However, several facts can be more
clearly interpreted in the setting of abstract Hilbert spaces, which will
be our set up in this section. With this particular example in mind,
in the following theorem, we can
think of $\{e_{j}^{k}\}$ being an orthonormal basis given by eigenfunctions
of an elliptic operator on $M$, and $d_{j}$ the corresponding 
multiplicities. However, we allow flexibility in grouping the eigenfunctions
in order to be able to also cover the case of operators
on compact Lie groups.

\begin{thm}\label{THM:inv-rem}
Let $\Hcal$ be a complex Hilbert space and let $\Hcal^{\infty}\subset \Hcal$ be a dense
linear subspace of $\Hcal$. Let $\{d_{j}\}_{j\in\N_{0}}\subset\N$ and let
$\{e_{j}^{k}\}_{j\in\N_{0}, 1\leq k\leq d_{j}}$ be an
orthonormal basis of $\Hcal$ such that
$e_{j}^{k}\in \Hcal^{\infty}$ for all $j$ and $k$. Let $H_{j}:={\rm span} \{e_{j}^{k}\}_{k=1}^{d_{j}}$,
and let $P_{j}:\Hcal\to H_{j}$ be the orthogonal projection.
For $f\in\Hcal$, we denote $$\widehat{f}(j,k):=(f,e_{j}^{k})_{\Hcal}$$ and let
$\widehat{f}(j)\in \ce^{d_{j}}$ denote the column of $\widehat{f}(j,k)$, $1\leq k\leq d_{j}.$
Let $T:\Hcal^{\infty}\to \Hcal$ be a linear operator.
Then the following
conditions are equivalent:
\begin{itemize}
\item[(A)] For each $j\in\ene_0$, we have $T(H_j)\subset H_j$. 
\item[(B)] For each $\ell\in\ene_0$ there exists a matrix 
$\sigma(\ell)\in\ce^{d_{\ell}\times d_{\ell}}$ such that for all $e_j^k$ 
$$
\widehat{Te_j^k}(\ell,m)=\sigma(\ell)_{mk}\delta_{j\ell}.
$$
\item[(C)]  If in addition, $e_j^k$ are in the domain of $T^*$ for all $j$ and $k$, then 
for each $\ell\in\ene_0 $ there exists a matrix 
$\sigma(\ell)\in\ce^{d_{\ell}\times d_{\ell}}$ such that
 \[\widehat{Tf}(\ell)=\sigma(\ell)\widehat{f}(\ell)\]
 for all $f\in\Hcal^{\infty}.$
\end{itemize}

The matrices $\sigma(\ell)$ in {\rm (B)} and {\rm (C)} coincide.

The equivalent properties {\rm (A)--(C)} follow from the condition 
\begin{itemize}
\item[(D)] For each $j\in\ene_0$, we have
$TP_j=P_jT$ on $\Hcal^{\infty}$.
\end{itemize}
If, in addition, $T$ extends to a bounded operator
$T\in{\mathscr L}(\Hcal)$ then {\rm (D)} is equivalent to {\rm (A)--(C)}.
\end{thm} 

Under the assumptions of Theorem \ref{THM:inv-rem}, we have the direct sum 
decomposition
\begin{equation}\label{EQ:sum}
\Hcal = \bigoplus_{j=0}^{\infty} H_{j},\quad H_{j}={\rm span} \{e_{j}^{k}\}_{k=1}^{d_{j}},
\end{equation}
and we have $d_{j}=\dim H_{j}.$
The two applications that we will consider will be with $\Hcal=L^{2}(M)$ for a
compact manifold $M$ with $H_{j}$ being the eigenspaces of an elliptic 
pseudo-differential operator $E$, or with $\Hcal=L^{2}(G)$ for a compact Lie group
$G$ with $$H_{j}=\textrm{span}\{\xi_{km}\}_{1\leq k,m\leq d_{\xi}}$$ for a
unitary irreducible representation $\xi\in[\xi_{j}]\in\widehat{G}$. The difference
is that in the first case we will have that the eigenvalues of $E$ corresponding
to $H_{j}$'s are all distinct, while in the second case the eigenvalues of the Laplacian
on $G$ for which $H_{j}$'s are the eigenspaces, may coincide.
In Remark \ref{REM:torus} we give an example of this difference for operators on
the torus $\Tn$.

In view of properties (A) and (C), respectively, an operator $T$ satisfying any of
the equivalent properties (A)--(C) in
Theorem \ref{THM:inv-rem}, will be called an {\em invariant operator}, or
a {\em Fourier multiplier relative to the decomposition
$\{H_{j}\}_{j\in\N_{0}}$} in \eqref{EQ:sum}.
If the collection $\{H_{j}\}_{j\in\N_{0}}$
is fixed once and for all, we can just say that $T$ is {\em invariant}
or a {\em Fourier multiplier}.

The family of matrices $\sigma$ will be
called the {\em matrix symbol of $T$ relative to the partition $\{H_{j}\}$ and to the
basis $\{e_{j}^{k}\}$.}
It is an element of the space $\Sigma$ defined by
\begin{equation}\label{EQ:Sigma1}
\Sigma=\{\sigma:\N_{0}\ni\ell\mapsto\sigma(\ell)\in \ce^{d_{\ell}\times d_{\ell}}\}.
\end{equation}
A criterion for the extendability of $T$ to ${\mathscr L}(\Hcal)$ in terms
of its symbol will be given in Theorem \ref{L2-abstract}.

For $f\in\Hcal$, in the notation of Theorem \ref{THM:inv-rem},
by definition we have
\begin{equation}\label{EQ:ser}
f=\sum_{j=0}^{\infty} \sum_{k=1}^{d_{j}} \widehat{f}(j,k) e_{j}^{k}
\end{equation}
with the convergence of the series in $\Hcal$.
Since $\{e^k_j\}_{j\geq 0}^{1\leq k\leq d_j}$ is a complete orthonormal 
system on  $\Hcal$, for all $f\in \Hcal$ we have the Plancherel formula
\beq \label{EQ:Plancherel}
\|f\|^2_{\Hcal}=\sum\limits_{j=0}^{\infty}\sum\limits_{k=1}^{d_j}|( f,e_j^k)|^2
=  \sum\limits_{j=0}^{\infty}\sum\limits_{k=1}^{d_j}|\widehat{f}(j,k)|^{2}
=\|\widehat{f}\|^{2}_{\ell^2(\N_{0},\Sigma)},
\eq
where we interpret $\widehat{f}\in\Sigma$ as an element of the space
\begin{equation}\label{EQ:aux3}
\ell^2(\N_{0,}\Sigma)=
\{h:\ene_0\rightarrow \prod\limits_d\ce^{d}: h(j)\in \ce^{d_j}\, \mbox{ and }\,\sum\limits_{j=0}^{\infty}\sum\limits_{k=1}^{d_j}|h(j,k)|^2<\infty\}, 
\end{equation}
and where we have written $h(j,k)=h(j)_k$. 
In other words, $\ell^2(\N_{0,}\Sigma)$ is the space of all $h\in\Sigma$ such that
$$
\sum\limits_{j=0}^{\infty}\sum\limits_{k=1}^{d_j}|h(j,k)|^2<\infty.
$$
We endow  $\ell^2(\N_{0},\Sigma)$ with the norm
\begin{equation}\label{EQ:aux4}
\|h\|_{\ell^2(\N_{0,}\Sigma)}:=\left(\sum\limits_{j=0}^{\infty}\sum\limits_{k=1}^{d_j}|h(j,k)|^2\right)^{\half}.
\end{equation}

We note that the matrix symbol $\sigma(\ell)$ depends 
not only on the partition \eqref{EQ:sum} but also
on the choice of the orthonormal basis.
Whenever necessary, we will indicate the dependance of $\sigma$ on the orthonormal 
basis by writing $(\sigma,\{e_j^k\}_{j\geq 0}^{1\leq k\leq d_j} )$ and we also will refer to 
$(\sigma,\{e_j^k\}_{j\geq 0}^{1\leq k\leq d_j} )$ as the {\em symbol} of $T$. 
Throughout this  section the orthonormal basis will be fixed and unless there is some 
risk of confusion the symbols will be denoted simply by $\sigma$.  
In the invariant language, 
as will be clear from the proof of Theorem \ref{THM:inv-rem},
we have that the transpose of the symbol,
$\sigma(j)^{\top}=T|_{H_{j}}$ is just the restriction of
$T$ to $H_{j}$, which is well defined in view of the property (A).

We will also sometimes 
write $T_{\sigma}$ to indicate that $T_{\sigma}$ is an operator corresponding to the 
symbol $\sigma $. It is clear from the definition that invariant operators are 
uniquely determined by their symbols. Indeed, if $T=0$ we obtain 
$\sigma=0$  for any choice of an orthonormal basis.  
Moreover, we note that by taking $j=\ell$ in (B) of Theorem \ref{THM:inv-rem} we obtain
the formula for the symbol:
\beq\label{symbinv}
\sigma(j)_{mk}=\widehat{Te_j^k}(j,m),
\eq
for all $1\leq k,m\leq d_j$. The formula (\ref{symbinv}) furnishes 
an explicit formula for the symbol in terms of the operator and the orthonormal basis. 
The definition of Fourier coefficients tells us that for invariant operators
we have 
\beq\label{symbinv2}
\sigma(j)_{mk}=({Te_j^k},e_j^m)_{L^2(M)}.
\eq
In particular,  for the identity operator $T=I$ we have $\sigma_{I}(j)=I_{d_{j}}$,
where $I_{d_{j}}\in \C^{{d_{j}}\times {d_{j}}}$ is the identity matrix.

Before proving Theorem  \ref{THM:inv-rem}, let us establish a formula relating symbols 
with respect to different orthonormal basis. 
If $\{e_{\alpha}\}$ and $\{f_{\alpha}\}$ are orthonormal bases of 
$\Hcal$, we consider the unitary operator $U$ determined by 
$U(e_{\alpha})=f_{\alpha}$. Then we have
\[
(Te_{\alpha}, e_{\beta})_{\Hcal}=(UTe_{\alpha}, Ue_{\beta})_{\Hcal}
=(UTU^*Ue_{\alpha}, Ue_{\beta})_{\Hcal}
=(UTU^*f_{\alpha}, f_{\beta})_{\Hcal}.
\]
If $(\sigma_{T}, \{e_{\alpha}\})$ denotes the symbol of $T$ with respect to the 
orthonormal basis $\{e_{\alpha}\}$ and $(\sigma_{UTU^*}, \{f_{\alpha}\})$ 
denotes the symbol of $UTU^*$ with respect to the orthonormal basis $\{f_{\alpha}\}$ 
we have obtained the relation
\beq\label{difsymb} (\sigma_{T}, \{e_{\alpha}\})=({\sigma_{UTU^*}}, \{f_{\alpha}\}).\eq
Thus, the equivalence relation of basis $\{e_{\alpha}\}\sim  \{f_{\alpha}\}$ given by 
a unitary operator $U$ induces
the equivalence relation on the set $\Sigma$ of symbols given by 
\eqref{difsymb}. In view of this,
we can also think of the symbol as an element of the space
$\Sigma/\sim$ with the equivalence relation given by
\eqref{difsymb}.

We make another remark concerning part (C) of Theorem \ref{THM:inv-rem}.
We use the condition that $e_j^k$ are in the domain ${\rm Dom}(T^*)$ of $T^*$ in showing the implication
(B) $\Longrightarrow$ (C). Since $e_j^k$'s give a basis in $\Hcal$, and are all contained in 
${\rm Dom}(T^*)$, it follows that ${\rm Dom}(T^*)$ is dense in $\Hcal$. In particular, 
by \cite[Theorem VIII.1]{r-s:vol1}, $T$ must be closable (in part (C)). These conditions are not restrictive for the further
analysis since they are satisfied in the natural applications of this paper.

The principal application of the notions above will be as follows,
except for in the sequel we will need more general operators $E$
unbounded on $\Hcal$. In order to distinguish from this general case,
in the following theorem we use the notation $\En$.

\begin{thm}\label{THM:inv-rem2}
Continuing with the notation of Theorem \ref{THM:inv-rem}, let 
$\En\in {\mathscr L}(\Hcal)$ be a linear continuous operator such
that $H_{j}$ are its eigenspaces:
$$\En e_{j}^{k}=\lambda_{j} e_{j}^{k}$$ for each $j\in\ene_{0}$ and all $1\leq k\leq d_{j}$.
Then equivalent conditions {\rm (A)--(C)} imply the property 
\begin{itemize}
\item[(E)]
For each $j\in\ene_{0}$ and $1\leq k\leq j$, we have
$T\En e_{j}^{k}=\En T e_{j}^{k},$
\end{itemize}
and if 
$\lambda_{j}\not=\lambda_{\ell}$ 
for $j\not=\ell$, then {\rm (E)} is equivalent to  
properties {\rm (A)--(C)}.

\smallskip
Moreover, if $T$ extends to a bounded operator
$T\in{\mathscr L}(\Hcal)$ then equivalent properties {\rm (A)--(D)} 
imply the condition
\begin{itemize}
\item[(F)]  $T\En=\En T$ on $\Hcal$,
\end{itemize}
and if also $\lambda_{j}\not=\lambda_{\ell}$ 
for $j\not=\ell$, then {\rm (F)} is equivalent to  {\rm (A)--(E)}.
\end{thm} 

For an operator $T=F(\En)$, when it is well-defined by the spectral calculus, we have
\begin{equation}\label{EQ:symbol-Fa}
\sigma_{F(\En)}(j)=F(\lambda_{j}) I_{d_{j}}.
\end{equation}
In fact, this is also well-defined then 
for a function $F$ defined on $\lambda_j$, with finite values which 
are e.g. $j$-uniformly bounded (also for non self-adjoint $E_o$).
We first prove Theorem \ref{THM:inv-rem}.

\begin{proof}[Proof of Theorem \ref{THM:inv-rem}]

(A) $\Longrightarrow$ (B).
If $T$  satisfies condition (A), we consider the matrix of 
$T|_{H_{j}}:H_{j}\rightarrow H_{j}$ with respect to the orthonormal basis 
$\{e^i_j:1\leq i\leq d_j\}$ of $H_j$ and denote it by $\beta(j)$. Then 
\[Te_j^k=\sum\limits_{i=1}^{d_j}\beta(j)_{ki}e_j^i.\]
Consequenlty, we have
\begin{align*} \widehat{Te_j^k}(\ell,m)=&(Te_j^k,e_{\ell}^m)
=\beta(j)_{km}\delta_{j\ell}
=\beta(\ell)_{km}\delta_{j\ell}.
\end{align*}
We take then $\sigma(\ell):=\beta(\ell)^{\top}$;
it belongs to $\ce^{d_{\ell}\times d_{\ell}}$ and satisfies (B).

\smallskip
(B) $\Longrightarrow$ (A).
Since $e_{j}^{k}\in\Hcal^{\infty}$, writing the series \eqref{EQ:ser} 
for $Te_j^k\in\Hcal$, we have
\begin{equation} \label{EQ:aux2}
Te_j^k=
\sum\limits_{\ell}\sum\limits_{m=1}^{d_{\ell}}\widehat{Te_j^k}(\ell,m)e_{\ell}^m
=\sum\limits_{\ell}\sum\limits_{m=1}^{d_{\ell}}\sigma(\ell)_{mk}\delta_{j\ell}e_{\ell}^m
=\sum\limits_{m=1}^{d_{\ell}}\sigma(j)_{mk}e_{j}^m\in \, H_j.
\end{equation}
Since $\{e^m_j:1\leq m\leq d_j\}$ spans $H_j$, we obtain (A).

\smallskip
(B) $\Longrightarrow$ (C).
We assume in addition that $e_j^k$ are in the domain of $T^*$ for all $j$ and $k$.
We also assume that for each $\ell\in\ene_0 $ 
there exists a matrix $\sigma(\ell)\in\ce^{d_{\ell}\times d_{\ell}}$ such that
\beq
\widehat{Te_j^k}(\ell,m)=\sigma(\ell)_{mk}\delta_{j\ell}.\label{eq21}
\eq
Now, if  $f\in\Hcal^{\infty}$, then $Tf\in\Hcal$, and by
 the inversion formula \eqref{EQ:ser} we have
$$
f=\sum\limits_{j=0}^{\infty}\sum\limits_{k=1}^{d_j}\widehat{f}(j,k) e_j^k.
$$
Now, using this and the fact that all $e_\ell^m$ are in the domain of $T^*$, we have 
\begin{align*} \widehat{Tf}(\ell,m)=&(Tf,e_\ell^m) \\
=&(f,T^* e_\ell^m)\\
=&\left(\sum\limits_{j=0}^{\infty}\sum\limits_{k=1}^{d_j} \widehat{f}(j,k)e_j^k,T^*e_\ell^m\right)\\
=&\sum\limits_{j=0}^{\infty}\sum\limits_{k=1}^{d_j}\widehat{f}(j,k) \left( Te_j^k,e_\ell^m\right)\\
=&\sum\limits_{j=0}^{\infty}\sum\limits_{k=1}^{d_j} \widehat{f}(j,k)\widehat{Te_j^k}(\ell,m)\\
=&\sum\limits_{j=0}^{\infty}\sum\limits_{k=1}^{d_j} \widehat{f}(j,k)\sigma(\ell)_{mk}\delta_{j\ell}\\
=&\sum\limits_{k=1}^{d_j} \sigma(\ell)_{mk}\widehat{f}(\ell,k),
\end{align*}
where we also used (\ref{eq21}).
Hence
$\widehat{Tf}(\ell)=\sigma(\ell)\widehat{f}(\ell),$
yielding (C).

\smallskip
(C) $\Longrightarrow$ (B). If $\widehat{Tf}(\ell)=\sigma(\ell)\widehat{f}(\ell)$,
then 
\begin{multline*}
\widehat{Te_j^k}(\ell,m)=\left(\sigma(\ell)\widehat{e_j^k}(\ell)\right)_{m}
=\sum\limits_{i=1}^{d_j}\sigma(\ell)_{mi}\widehat{e_j^k}(\ell,i)
=\sum\limits_{i=1}^{d_j}\sigma(\ell)_{mi}\delta_{j\ell}\delta_{ki}
=\sigma(\ell)_{mk}\delta_{j\ell},
\end{multline*}
which gives (B), even without any assumptions on $T^*$.

\smallskip
(D) $\Longrightarrow$ (A). We 
take $f\in H_{j}.$ Then $P_{j} f\in H_{j}$  since $P_{j} f=f$, so that by assumption (D)
we have
$$
Tf=TP_{j}f=P_{j}T f\in H_{j},
$$
implying (A).

\smallskip
(A) $\Longrightarrow$ (D).
For this part we assume in addition that $T$ extends to a
bounded operator $T\in {\mathscr L}(\Hcal)$.
First, we show that this together with (A) implies that $T(H_j^{\bot})$ is orthogonal to $H_j$.
For $g\in H_j^{\bot}$, we can write 
\[
g=\sum\limits_{\ell\neq j}\sum\limits_{k=1}^{d_{\ell}}(g,e_{\ell}^k)e_{\ell}^k
\]
with the convergence in $\Hcal$, so that
\[
Tg=\sum\limits_{\ell\neq j}\sum\limits_{k=1}^{d_{\ell}}(g,e_{\ell}^k)Te_{\ell}^k
\] 
with the convergence in $\Hcal$ due to the boundedness of $T$ on $\Hcal$.
Since by (A) we have
$Te_{\ell}^k\in H_{\ell}\subset H_j^{\bot} $ for $\ell\neq j$ we conclude that 
$Tg$ is orthogonal to $H_{j}$.

Let now $f\in \Hcal^{\infty}$. Writing
$f=f_{1}+f_{2}$ with $f_{1}:=P_{j}f$ so that 
$f_{1}\in H_{j}$ and $f_{2}\in H_{j}^{\bot}$ are both in $\Hcal^{\infty}$, we have
$$
P_{j}Tf=P_{j}Tf_{1}+P_{j}Tf_{2}=Tf_{1}=TP_{j}f,
$$
since the proved claim
$P_{j}f_{2}=0$ implies that $P_{j}Tf_{2}=0$.
\end{proof}

We now continue with the proof of  Theorem \ref{THM:inv-rem2} when the basis
$e_{j}^{k}$ corresponds to the eigenvectors of an operator
$\En\in {\mathcal L}(\Hcal).$

\begin{proof}[Proof of Theorem \ref{THM:inv-rem2}]

(A) $\Longrightarrow$ (E).
Let us fix some $e_j^k$. By condition (A) we can write 
\[
Te_j^k=\sum\limits_{i=1}^{d_j}\alpha_{i}e_j^i
\]
for some constants $\alpha_{i}$.
Then
\begin{multline*} \En T e_j^k=\En \sum\limits_{i=1}^{d_j}\alpha_{i}e_j^i
=\sum\limits_{i=1}^{d_j}\alpha_{i}\lambda_j e_j^i
=\lambda_j\sum\limits_{i=1}^{d_j}\alpha_{i}e_j^i
=\lambda_j Te_j^k
=T\lambda_je_j^k
=T \En e_j^k,
\end{multline*}
which shows (E).

\smallskip
(E) $\Longrightarrow$ (A).
%
We note that 
it is enough to prove that $Te_j^k\in H_j$ since $\{e_j^k:1\leq k\leq d_j\}$ forms a basis of the 
finite dimensional space $H_j$. We can assume that 
$Te_j^k\neq 0$ since otherwise there is nothing to prove.
We recall that $\En e_j^k=\lambda_j e_j^k$. Using property (E), we have
\[
\lambda_j Te_j^k=T\En e_j^k = \En T {e_j^k}.
\]
Hence
$Te_j^k\in \Hcal$ is a non-zero eigenvector of $\En$ corresponding to the eigenvalue 
$\lambda_j$. Consequently, since $H_{j}$ are maximal eigenspaces
corresponding to $\lambda_{j}$, we must have $Te_j^k\in H_j.$ 

\smallskip
(E) $\Longrightarrow$ (F).
Since we have already shown that 
(A)--(C) always imply (E), it is enough to prove that (E) implies (F)
under the additional assumption that $T\in {\mathscr L}(\Hcal)$.

Let us write $S:=\En\circ T, D:=T\circ \En$ and let $f\in \Hcal$. 
Under the assumptions both $S$ and $D$ are bounded on $\Hcal$,
and hence the formula \eqref{EQ:ser} implies
\[
Sf=\lim\limits_N\sum\limits_{j=0}^{N}\sum\limits_{k=1}^{d_j}(f,e_j^k)Se_j^k
=\lim\limits_N\sum\limits_{j=0}^{N}\sum\limits_{k=1}^{d_j}(f,e_j^k)De_j^k
=Df,
\]
with the convergent series in $\Hcal$.

\smallskip
(F) $\Longrightarrow$ (A). 
We note that we require $T\in {\mathscr L}(\Hcal)$ in order 
for $T\En$ and $\En T$ to make sense on $\Hcal$.
It is clear that (F) implies (E),
and under the additional assumption that 
$\lambda_{j}\not=\lambda_{\ell}$ 
for $j\not=\ell$ we already know that (A)--(C) and (E) are equivalent.
If $T$ is bounded on $\Hcal$, then they are also equivalent to (D).
\end{proof}

We have the following criterion for the extendability of
a densely defined invariant operator $T:\Hcal^{\infty}\to \Hcal$ to ${\mathscr L}(\Hcal)$,
which was an additional hypothesis for properties (D) and (F).
In the statements below we fix a partition into $H_{j}$'s as in \eqref{EQ:sum}
and the invariance refers to it.

\begin{thm}\label{L2-abstract} 
An invariant linear operator $T:\Hcal^{\infty}\to \Hcal$ extends to a bounded
operator from $\Hcal$ to $\Hcal$ if and only if its symbol $\sigma$ satisfies
$\sup\limits_{\ell\in\N_{0}}\|\sigma(\ell)\|_{{\mathscr L}(H_{\ell})}<\infty.$
Moreover, denoting this extension also by $T$, we have
\[
\| T\|_{{\mathscr L}(\Hcal)} =\sup\limits_{\ell\in \N_{0}}\|\sigma(\ell)\|_{{\mathscr L}(H_{\ell})}.
\]
\end{thm} 
\begin{proof}  
We will often abbreviate writing $\|\sigma(\ell)\|_{op}:=\|\sigma(\ell)\|_{{\mathscr L}(H_{\ell})}$.
Let us first suppose that 
$\|\sigma(\ell)\|_{op}\leq C$ for all $\ell\in\N_{0}$. 
By the Plancherel formula \eqref{EQ:Plancherel}
we have
\begin{align*} \|Tf\|_{\Hcal}^2=&\|\widehat{Tf}\|_{\ell^2(\ene_0,\Sigma)}^2\\
=& \sum\limits_{\ell}\|\widehat{Tf}(\ell)\|_{\ell^2(\ce^{d_{\ell}})}^2\\
=& \sum\limits_{\ell}\|\sigma(\ell)\widehat{f}(\ell)\|_{\ell^2(\ce^{d_{\ell}})}^2\\
\leq& \sum\limits_{\ell}\|\sigma(\ell)\|^{2}_{op}\|\widehat{f}(\ell)\|_{\ell^2(\ce^{d_{\ell}})}^2\\
\leq& \sup\limits_{\ell}\|\sigma(\ell)\|^{2}_{op}
\sum\limits_{\ell}\|\widehat{f}(\ell)\|_{\ell^2(\ce^{d_{\ell}})}^2\\
=& \p{\sup\limits_{\ell}\|\sigma(\ell)\|_{op}}^{2}\|f\|^{2}_{\Hcal}.
\end{align*}

Conversely, let us suppose that $T$ is bounded on $\Hcal$. 
Then there exists a constant $C>0$ such that $\|Tf\|_{\Hcal}\leq C$ for all
 $f$ such that $\|f\|_{\Hcal}=1$. 
We can take $C:=\| T\|_{{\mathscr L}(\Hcal)}$. Hence 
\[T|_{H_{j}}:H_{j}\rightarrow H_{j}\]
is bounded and $\|T|_{H_{j}}\|_{\mathscr{L}(H_{j})}\leq C$. On the other hand, let
$\beta(j)$ denote the matrix of $T|_{H_{j}}:H_{j}\rightarrow H_{j}$ with respect to the orthonormal basis $\{e^i_j:1\leq i\leq d_j\}$ of $H_j$ as in the proof of Part
(A) implies (B) in Theorem \ref{THM:inv-rem}. 
We consider an unitary operator $U:H_{j}\to \ce^{d_{j}}$ which defines coordinates in $\ce^{d_{j}}$ 
of vectors in $H_{j}$  with respect to the orthonormal basis $\{e^k_{j}:1\leq k\leq d_{j}\}$ of $H_{j}$. We also consider the operator 
$A(j):\ce^{d_{j}}\to \ce^{d_{j}}$ induced by the matrix $\beta(j)$. Then 
\[T|_{H_{j}}=U^*A(j)U,\] 
and 
\[\|\sigma(j)\|_{op}=\|\beta(j)\|_{op}=\|A(j)\|_{op}=\|T|_{H_{j}}\|_{\mathscr{L}(H_{j})}\leq C,\]
completing the proof.
\end{proof}

We also record the formula for the symbol of the composition of two invariant
operators:

\begin{prop}\label{comp1-abstract} 
If $S,T:\Hcal^{\infty}\to \Hcal$ are invariant operators with respect to the same 
orthonormal partition, and such that the domain of $S\circ T$ contains $\Hcal^{\infty}$,
then $S\circ T:\Hcal^{\infty}\to \Hcal$ is also 
invariant with respect to the same partition. Moreover, if $\sigma_S$ denotes the symbol of 
$S$ and $\sigma_T$ denotes the symbols of $T$ 
with respect to  the same orthonormal basis then 
\[\sigma_{S\circ T}=\sigma_S\sigma_T,\]
i.e. $\sigma_{S\circ T}(j)=\sigma_S(j)\sigma_T(j)$ for all $j\in\N_{0}.$
\end{prop} 
\begin{proof}
Recalling the definition of the composition of densely defined operators,
the domain of $S\circ T$ is the space of functions $f$ in the domain of $T$ such that
$Tf$ is in the domain of $S$, in which case we set
$(S\circ T)f=S(Tf)$. The assumption says that we are in the position to use
Theorem \ref{THM:inv-rem}.
Applying the condition (C) of Theorem \ref{THM:inv-rem} repeatedly, we have
$$
\widehat{(S\circ T)f}(j)=\widehat{S(Tf)}(j)
=\sigma_S(j)\widehat{Tf}(j)\\
=\sigma_S(j)\sigma_T(j)\widehat{f}(j),
$$
so that $S\circ T$ is invariant by Part (C) of Theorem \ref{THM:inv-rem}.
\end{proof}

We now show another application of 
the above notions to give a characterisation of Schatten classes of invariant operators in
terms of their symbols. 

\begin{thm}\label{schchr-abstract} Let $0<r<\infty$. 
An invariant operator $T\in {\mathscr L}(\Hcal)$ with symbol 
$\sigma$ is in the Schatten class
$S_r(\Hcal)$ 
if and only if  $$\sum\limits_{\ell=0}^{\infty}\|\sigma(\ell)\|_{S_r(H_{\ell})}^r<\infty.$$ 
Moreover
\begin{equation}\label{EQ:Sch1}
\|T\|_{S_r(\Hcal)}=\p{\sum\limits_{\ell=0}^{\infty}\|\sigma(\ell)\|_{S_r(H_{\ell})}^r}^{1/r}.
\end{equation}
In particular, if $T$ is in the trace class
$S_1(\Hcal)$, then we have the trace formula
\begin{equation}\label{EQ:Sch2}
\Tr(T)=\sum\limits_{\ell=0}^{\infty}\Tr(\sigma(\ell)).
\end{equation}
\end{thm}

\begin{proof} 

First, we claim that Schatten classes of 
invariant operators can be characterised
in terms of the projections to the eigenspaces $H_{\ell}$: 
\begin{equation}\label{EQ:T-lem}
\|T\|_{S_r(\Hcal)}^r=\sum\limits_{\ell=0}^{\infty}\|T|_{H_{\ell}}\|_{S_r(H_{\ell})}^r.
\end{equation}
Let us prove \eqref{EQ:T-lem}.
Since $$\|T\|_{S_r}=\||T|\|_{S_r}$$ we can assume without loss of
generality that $T$ is positive definite. We first observe that
$\lambda$ is an eigenvalue (singular value) of $T$ if and only if $\lambda$ is an eigenvalue (singular value) of $T|_{H_{\ell(\lambda)}} $ for some $\ell(\lambda)$. Indeed, if $\lambda$ is an eigenvalue of $T$ there exists $\varphi_{\lambda}\in \Hcal\backslash\{0\}$ such that 
$T\varphi_{\lambda}=\lambda\varphi_{\lambda}$. 
Using Part (D) of Theorem \ref{THM:inv-rem}, we get that 
$$TP_{\ell}\varphi_{\lambda}=\lambda P_{\ell}\varphi_{\lambda}$$ holds for every $\ell$.
Since $\varphi_{\lambda}\not=0$, there 
exists $\ell(\lambda)$ such that $P_{\ell(\lambda)}\varphi_{\lambda}\not=0$.
Consequently, $\lambda$ is the eigenvalue of $T|_{H_{\ell(\lambda)}}=TP_{\ell(\lambda)}$. 
Conversely, since $T(H_{\ell(\lambda)})\subset H_{\ell(\lambda)}$, 
an eigenvalue of $T|_{H_{\ell(\lambda)}}$ is also an eigenvalue of $T$. 
Therefore, we obtain \eqref{EQ:T-lem}.

Now, given \eqref{EQ:T-lem}, to prove \eqref{EQ:Sch1}, it is enough to check that 
\beq
\label{syop}\|\sigma(\ell)\|_{S_r(H_{\ell})}=\|T|_{H_{\ell}}\|_{S_r(H_{\ell})}.
\eq
To prove (\ref{syop}) we consider an unitary operator 
$U:H_{\ell}\to \ce^{d_{\ell}}$ which defines coordinates in 
$\ce^{d_{\ell}}$ of functions in $H_{\ell}$  with respect to the orthonormal 
basis $\{e^k_{\ell}:1\leq k\leq d_{\ell}\}$ of $H_{\ell}$. 
We also consider the operator $A(\ell):\ce^{d_{\ell}}\to \ce^{d_{\ell}}$ 
induced by the matrix $(\sigma_{T}(\ell))^{\top}$. Then 
\[T|_{H_{\ell}}=U^*A(\ell)U,\] 
 and basic properties of Schatten quasinorms imply that
 \[\|T|_{H_{\ell}}\|_{S_r(H_{\ell})}=\|A(\ell)\|_{S_r({\ce^{d_{\ell}}})}=\|\sigma(\ell)\|_{S_r},\]
 completing the proof of (\ref{syop}) and of \eqref{EQ:Sch1}.

Finally, let us prove \eqref{EQ:Sch2} for operators in the trace class $S_{1}(\Hcal)$.
Since the trace $\Tr(T)$ does not depend on the choice of the orthonormal basis in
$\Hcal$, using property (C) and formula \eqref{EQ:aux2}, we can write 
\begin{multline*}
\Tr(T)=
\sum\limits_{\ell}\sum\limits_{k=1}^{d_{\ell}}(T e_{\ell}^k,e_{\ell}^k)
=\sum\limits_{\ell}\sum\limits_{k=1}^{d_{\ell}}\sum\limits_{m=1}^{d_{\ell}}\sigma(\ell)_{mk}(e_{\ell}^m,e_{\ell}^k)\\
=\sum\limits_{\ell}\sum\limits_{k=1}^{d_{\ell}}\sum\limits_{m=1}^{d_{\ell}}\sigma(\ell)_{mk}\delta_{mk}
=\sum\limits_{\ell}\sum\limits_{k=1}^{d_{\ell}}\sigma(\ell)_{kk}
=\sum\limits_{\ell}\Tr(\sigma(\ell)),
\end{multline*}
completing the proof.
\end{proof}

\begin{rem}\label{REM:torus}
We note that the membership in ${\mathscr L}(\Hcal)$ and in the Schatten classes
$S_{r}(\Hcal)$ does not depend on the decomposition of $\Hcal$ into subspaces $H_{j}$
as in \eqref{EQ:sum}. However, the notion of invariance does depend on it.
For example, let $\Hcal=L^{2}(\Tn)$ for the $n$-torus $\Tn=\Rn/\Zn$. 
Choosing 
$$H_{j}={\rm span}\{ e^{2\pi{\rm  i} j\cdot x} \}, \quad j\in\Zn,$$ 
we recover
the construction of Section \ref{SEC:Lie-groups} on compact Lie groups
and moreover, invariant operators
with respect to $\{H_{j}\}_{j\in\Zn}$ are the translation invariant operators on the torus
$\Tn$. However, to recover the construction of Section \ref{SEC:invariant}
on manifolds, we take
$\widetilde{H_{\ell}}$ to be the eigenspaces of the Laplacian $E$ on $\Tn$, so that
$$
\widetilde{H_{\ell}}=\bigoplus_{|j|^{2}=\ell} H_{j}=
{\rm span}\{e^{2\pi{\rm  i} j\cdot x}:\; j\in\Zn \textrm{ and }
|j|^{2}=\ell\},
\quad \ell\in\N_{0}.
$$ 
Then translation invariant operators on $\Tn$, i.e. operators invariant
relative to the partition $\{H_{j}\}_{j\in\Zn}$, are also invariant relative to the partition
$\{\widetilde{H_{\ell}}\}_{\ell\in\N_{0}}$ 
(or relative to the Laplacian, in terminology of Section \ref{SEC:invariant}).
If we have information on the eigenvalues of
$E$, like we do on the torus, we may sometimes also
recover invariant operators relative to the partition
$\{\widetilde{H_{\ell}}\}_{\ell\in\N_{0}}$ as linear combinations of translation invariant
operators composed with phase shifts and complex conjugation.
\end{rem}

\section{Fourier analysis associated to an elliptic operator}
\label{SEC:Fourier}

Our main application will be to study operators on compact manifolds, so we start this
section by describing the discrete Fourier series associated to an 
elliptic positive pseudo-differential operator as an adaptation of the construction in
Section \ref{SEC:abstract}. In order to fix the notation for the rest of the paper
we may give some explicit expressions for notions of Section \ref{SEC:abstract}
in the present setting.

Let $M$ be a compact smooth manifold of dimension $n$ without boundary, endowed with a fixed volume $dx$.  
 We denote by $\Psi^{\nu}(M)$ the H\"ormander class of pseudo-differential 
 operators of order $\nu\in\er$,
 i.e. operators which, in every coordinate chart, are operators in H\"ormander classes 
 on $\Rn$ with symbols
 in $S^\nu_{1,0}$, see e.g. \cite{shubin:r} or \cite{rt:book}.
 In this paper we will be using the class  $\Psi^{\nu}_{cl}(M)$ of classical operators, i.e. operators
 with symbols having (in all local coordinates) an asymptotic expansion of the symbol in
 positively homogeneous components (see e.g. \cite{Duis:BK-FIO-2011}).
 Furthermore, we denote by $\Psi_{+}^{\nu}(M)$ the class of positive definite operators in 
 $\Psi^{\nu}_{cl}(M)$,
 and by $\Psi_{e}^{\nu}(M)$ the class of elliptic operators in $\Psi^{\nu}_{cl}(M)$. Finally, 
 $$\Psi_{+e}^{\nu}(M):=\Psi_{+}^{\nu}(M)\cap \Psi_{e}^{\nu}(M)$$ 
 will denote the  class of classical positive elliptic 
 pseudo-differential operators of order $\nu$.
 We note that complex powers of such operators are well-defined, see e.g.
 Seeley \cite{Seeley:complex-powers-1967}. 
 In fact, all pseudo-differential operators considered in 
 this paper will be classical, so we may omit explicitly mentioning it every time, but we note
 that we could equally work with general operators in $\Psi^{\nu}(M)$ since their
 powers have similar properties, see e.g. \cite{Strichartz:functional-calculus-AJM-1972}.
 
We now associate a discrete Fourier analysis to the operator 
$E\in\Psi_{+e}^{\nu}(M)$ inspired 
by those constructions
considered by Seeley (\cite{see:ex}, \cite{see:exp}), see also 
Greenfield and Wallach \cite{Greenfield-Wallach:hypo-TAMS-1973}. 
However, we adapt it to our purposes and in the sequel also
prove several auxiliary statements concerning the
eigenvalues of $E$ and their multiplicities, useful to us in the subsequent analysis.
In general, the construction below is exactly the one appearing in
Theorem \ref{THM:inv-rem}.

The eigenvalues of $E$ (counted without multiplicities)
form a sequence $\{\lambda_j\}$ which we order so that
\begin{equation}\label{EQ:lambdas}
0=\lambda_{0}<\lambda_{1}<\lambda_{2}<\cdots.
\end{equation}
For each eigenvalue $\lambda_j$, there is
the corresponding finite dimensional eigenspace $H_j$ of functions on $M$, which are smooth due to the 
ellipticity of $E$. We set 
$$
d_j:=\dim H_j, 
\textrm{ and } H_0:=\ker E, \; \lambda_0:=0.
$$
We also set $d_{0}:=\dim H_{0}$. Since the operator $E$ is elliptic, it is Fredholm,
hence also $d_{0}<\infty$ (we can refer to \cite{Atiyah:global-aspects-1968}, \cite{ho:apde2} for
various properties of $H_{0}$ and $d_{0}$).

We fix  an orthonormal basis of $L^2(M)$ consisting of eigenfunctions of $E$:
\beq\label{fam}\{e^k_j\}_{j\geq 0}^{1\leq k\leq d_j},\eq 
where $\{e^k_j\}^{1\leq k\leq d_j}$ is an orthonormal basis of $H_j$. 
Let $P_j:L^2(M)\rightarrow H_j$ be the corresponding projection. 
We shall denote by $(\cdot,\cdot)$ the inner product of $L^2(M)$. 
 We observe that we have
 \[P_jf=\sum\limits_{k=1}^{d_j}(f,e_j^k) e_j^k,\]
for $f\in L^2(M)$. The `Fourier' series takes the form 
\[f=\sum\limits_{j=0}^{\infty}\sum\limits_{k=1}^{d_j}(f,e_j^k )e_j^k,\]
for each $f\in L^2(M)$.
The Fourier coefficients of $f\in L^2(M)$ with respect to the orthonormal basis $\{e^k_j\}$ 
will be denoted by 
\begin{equation}\label{EQ:F-coeff}
(\efee f)(j,k):=\widehat{f}(j,k):=(f,e_j^k).
\end{equation}
We will call the collection of $\widehat{f}(j,k)$ the {\em Fourier coefficients of $f$ relative to $E$},
or simply the {\em Fourier coefficients of $f$}.

\smallskip
Since $\{e^k_j\}_{j\geq 0}^{1\leq k\leq d_j}$ forms a complete orthonormal 
system in  $L^2(M)$, for all $f\in L^2(M)$ we have the Plancherel formula
\eqref{EQ:Plancherel}, namely,
\beq \label{EQ:Plancherel2}
\|f\|^2_{L^{2}(M)}=\sum\limits_{j=0}^{\infty}\sum\limits_{k=1}^{d_j}|( f,e_j^k)|^2
=  \sum\limits_{j=0}^{\infty}\sum\limits_{k=1}^{d_j}|\widehat{f}(j,k)|^{2}
=\|\widehat{f}\|^{2}_{\ell^2(\ene_0,\Sigma)},
\eq
where the space $\ell^2(\ene_0,\Sigma)$ and its norm are as in
\eqref{EQ:aux3} and \eqref{EQ:aux4}.

We can think of $\efee=\efee_{M}$ as of the Fourier transform being an isometry
from $L^2(M)$ into 
$\ell^2(\ene_0,\Sigma)$. The inverse of this Fourier transform can be then expressed by 
\beq\label{Fourier1inv}
(\efee^{-1}{h})(x)=\sum\limits_{j= 0}^{\infty}\sum\limits_{k=1}^{d_j}h(j,k)e_j^k(x).
\eq
If $f\in L^2(M)$, we also write
\[\widehat{f}(j)=\left(\begin{array}{c}\widehat{f}(j,1)\\
\vdots\\
\widehat{f}(j,d_j)\end{array} \right)\in\ce^{d_j},\]
thus thinking of the Fourier transform always as a column vector.
In particular, we think of 
\[\widehat{e_j^k}(\ell)=\left(\widehat{e_j^k}(\ell,m)\right)_{m=1}^{d_{\ell}}\]
as of a column, and we notice that
\beq
\widehat{e_j^k}(\ell,m)=\delta_{j\ell}\delta_{km}.
\label{eqdeltas}
\eq
Smooth functions on $M$ can be characterised by
\begin{align}\label{EQ:smooth}
f\in C^{\infty}(M)  & \Longleftrightarrow 
\forall N \; \exists C_{N}: \; |\widehat{f}(j,k)|\leq C_{N} (1+\lambda_{j})^{-N}
\textrm{ for all } j, k \\ \nonumber
& \Longleftrightarrow 
\forall N \; \exists C_{N}: \; |\widehat{f}(j)|\leq C_{N} (1+\lambda_{j})^{-N}
\textrm{ for all } j,
\end{align}
where $|\widehat{f}(j)|$ is the norm of the vector $\widehat{f}(j)\in\C^{d_{j}}.$
The implication `$\Longleftarrow$' here is immediate, while `$\Longrightarrow$'
follows from the Plancherel formula \eqref{EQ:Plancherel} and the fact that
for $f\in C^{\infty}(M)$ we have $(I+E)^{N}f\in L^{2}(M)$ for any $N$.

\smallskip
For $u\in \Dcal'(M)$, we denote its Fourier coefficient 
$$\widehat{u}(j,k):=u(\overline{e_{j}^{k}}),$$ and by duality, the space of distributions
can be characterised by
$$
f\in \Dcal'(M)   \Longleftrightarrow 
\exists M \; \exists C: \; |\widehat{u}(j,k)|\leq C(1+\lambda_{j})^{M}
\textrm{ for all } j, k.
$$

We will denote by $H^{s}(M)$ the usual Sobolev space over $L^{2}$ on $M$.
This space can be defined in local coordinates or, by the fact that 
$E\in \Psi^{\nu}_{+e}(M)$ is positive and elliptic with $\nu>0$, it can be
characterised by
\begin{multline}\label{EQ:Sob-char}
 f\in H^{s}(M) \Longleftrightarrow (I+E)^{s/\nu} f\in L^{2}(M)
 \Longleftrightarrow
 \{(1+\lambda_{j})^{s/\nu}\widehat{f}(j)\}_{j} \in \ell^{2}(\N_{0},\Sigma) \\
 \Longleftrightarrow
 \sum\limits_{j=0}^{\infty}\sum\limits_{k=1}^{d_j}
 (1+\lambda_{j})^{2s/\nu}|\widehat{f}(j,k)|^{2}<\infty.
\end{multline}
the last equivalence following from the Plancherel formula
\eqref{EQ:Plancherel}.
For the characterisation of analytic functions (on compact manifolds $M$) 
we refer to Seeley \cite{see:exp}.


\section{Invariant operators and symbols on compact manifolds}
\label{SEC:invariant}

We now discuss an application of a notion of an invariant operator and of its symbol
from Theorem \ref{THM:inv-rem} in the case of 
$\Hcal=L^{2}(M)$ and $\Hcal^{\infty}=C^{\infty}(M)$ and describe its basic
properties. 
We will consider operators $T$ densely defined on $L^{2}(M)$, and we will be making
a natural assumption that their domain contains $C^{\infty}(M)$.
We also note that while in Theorem \ref{THM:inv-rem2} it was assumed that the
operator $\En$ is bounded on $\Hcal$, this is no longer the case for the operator $E$ here.
Indeed, an elliptic  pseudo-differential operator $E\in\Psi_{+e}^{\nu}(M)$ of order
$\nu>0$ is not bounded on $L^{2}(M)$. 

Moreover, we do not want to assume that $T$ extends to a bounded operator on 
$L^{2}(M)$
to obtain analogues of properties (D) and (F) in Section \ref{SEC:abstract},
because this is too restrictive from the point of view of differential operators.
Instead, we show that in the present setting it is enough to assume that
$T$ extends to a continuous operator on $\Dcal'(M)$ to reach the same
conclusions.

So, we combine the statement of Theorem \ref{THM:inv-rem} and the necessary 
modification of Theorem \ref{THM:inv-rem2} to the setting of
Section \ref{SEC:Fourier} as follows.

We also remark that Part (iv) of the following theorem provides a correct formulation for
a missing assumption in \cite[Theorem 3.1, (iv)]{Delgado-Ruzhansky:CRAS-kernels}.

\begin{thm}\label{THM:inv}
Let $M$ be a closed manifold and 
let $T:\cinfm\to L^{2}(M)$
be a linear operator.
Then the following
conditions are equivalent:
\begin{itemize}
\item[(i)] For each $j\in\ene_0$, we have $T(H_j)\subset H_j$. 
\item[(ii)]
For each $j\in\ene_{0}$ and $1\leq k\leq j$, we have
$TE e_{j}^{k}=ET e_{j}^{k}.$
\item[(iii)] For each $\ell\in\ene_0$ there exists a matrix 
$\sigma(\ell)\in\ce^{d_{\ell}\times d_{\ell}}$ such that for all $e_j^k$ 
\beq\label{invadef}\widehat{Te_j^k}(\ell,m)=\sigma(\ell)_{mk}\delta_{j\ell}.\eq
\item[(iv)]  If, in addition, the domain of $T^*$ contains $C^\infty(M)$, then for each $\ell\in\ene_0 $ there exists a matrix 
$\sigma(\ell)\in\ce^{d_{\ell}\times d_{\ell}}$ such that
 \[\widehat{Tf}(\ell)=\sigma(\ell)\widehat{f}(\ell)\]
 for all $f\in\cinfm.$
\end{itemize}
The matrices $\sigma(\ell)$ in {\rm (iii)} and {\rm (iv)} coincide.

\smallskip
If $T$
extends to a linear continuous operator
$T:\Dcal'(M)\rightarrow \Dcal'(M)$ then
the above properties are also equivalent to the following ones:
\begin{itemize}
\item[(v)] For each $j\in\ene_0$, we have
$TP_j=P_jT$ on $C^{\infty}(M)$.
\item[(vi)]  $TE=ET$ on $L^{2}(M)$.
\end{itemize}
\end{thm} 

If any of the equivalent conditions (i)--(iv) of Theorem \ref{THM:inv} are satisfied, 
we say that the operator $T:\cinfm\rightarrow L^{2}(M)$ 
is {\em invariant (or is a Fourier multiplier) relative to $E$}.
We can also say 
that $T$ is $E$-invariant or is an $E$-multiplier. 
This recovers the notion of invariant operators 
given by Theorem \ref{THM:inv-rem}, with respect to the partitions
$H_{j}$'s in \eqref{EQ:sum} which are fixed being the eigenspaces of $E$.
When there is no risk of confusion we will just refer to such kind of operators 
as invariant operators or as multipliers.
It is clear from (i) that the operator $E$ itself or functions of $E$ defined
by the functional calculus are invariant relative to $E$. 

We note that the boundedness of $T$ on $L^{2}(M)$ needed for conditions
(D) and (F) in Theorem \ref{THM:inv-rem} and in Theorem \ref{THM:inv-rem2}
is now replaced by the condition that $T$ is continuous on $\Dcal'(M)$ which
explored the additional structure of $L^{2}(M)$ and allows application to
differential operators.

We call $\sigma$ in (iii) and (iv) the {\em matrix symbol of $T$} or simply 
the {\em symbol}. It is an element of the space $\Sigma=\Sigma_{M}$ defined by
\begin{equation}\label{EQ:Sigma}
\Sigma_M:=\{\sigma:\N_{0}\ni\ell\mapsto\sigma(\ell)\in \ce^{d_{\ell}\times d_{\ell}}\}.
\end{equation}
Since the expression for the symbol depends only on the basis $e_{j}^{k}$ and not
on the operator $E$ itself, this notion coincides with the symbol defined in
Theorem \ref{THM:inv-rem}.

Let us comment on several conditions in Theorem \ref{THM:inv} in this setting.
Assumptions (v) and (vi) are stronger than those in (i)--(iv).
On one hand, clearly (vi) contains (ii). On the other hand, 
as we will see in the proof, assumption (v) implies (i) without the 
additional hypothesis that $T$ is continuous on $\Dcal'(M)$.

In analogy to the strong commutativity in (v), {\em if $T$ is continuous on $\Dcal'(M)$},
so that all the assumptions (i)--(vi) are equivalent, we may say that
$T$ is {\em strongly invariant relative to $E$} in this case.

The expressions in (vi) make sense as both sides are defined (and even continuous) on
$\Dcal'(M)$. 

We also note that without additional assumptions, it is known from the
general theory of densily defined operators on Hilbert spaces that conditions
(v) and (vi) are generally not equivalent, see e.g.
Reed and Simon \cite[Section VIII.5]{r-s:vol1}.
If $T$ is a differential operator, the additional assumption of continuity
on $\Dcal'(M)$ for 
parts (v) and (vi) is satisfied.  In
\cite[Section 1, Definition 1]{Greenfield-Wallach:hypo-TAMS-1973} 
Greenfield and Wallach
called a differential operator $D$ to be an $E$-invariant operator if  $ED=DE$,
which is our condition (vi). However, Theorem  \ref{THM:inv} describes more
general operators as well as reformulates them in the form of Fourier multipliers
that will be explored in the sequel.

There will be several useful classes of symbols, in particular the moderate growth
class
\begin{equation}\label{EQ:sym-tempe}
{\mathcal S}'(\Sigma):=\{\sigma\in\Sigma:
\exists N, C \textrm{ such that }
\|\sigma(\ell)\|_{op}\leq C(1+\lambda_{\ell})^{N} \; \forall \ell\in\N_{0}
\},
\end{equation}
where $$ \|\sigma(\ell)\|_{op}=\|\sigma(\ell)\|_{{\mathscr L}(H_{\ell})}$$ 
denotes the matrix multiplication
operator norm with respect to $\ell^2(\ce^{d_{\ell}})$.

In the case when $M$ is a compact Lie group and $E$ is a Laplacian on $G$,
left-invariant operators on $G$, i.e. operators commuting with the left action of $G$,
are also invariant relative to $E$ in the sense of Theorem \ref{THM:inv};
this will be shown in Proposition \ref{scheq1}  after we investigate in 
Section \ref{SEC:Lie-groups} the relation
between the symbol in Theorem \ref{THM:inv} and matrix symbols of operators
on compact Lie groups. However, we need an adaptation of the above construction
since the natural decomposition into $H_{j}$'s in \eqref{EQ:sum} may in general violate
the condition \eqref{EQ:lambdas}.

As in Section \ref{SEC:abstract} since the notion of the symbol depends only
on the basis, for the identity operator $T=I$ we have $$\sigma_{I}(j)=I_{d_{j}},$$
where $I_{d_{j}}\in \C^{I_{d_{j}}\times I_{d_{j}}}$ is the identity matrix, and 
for an operator $T=F(E)$, when it is well-defined by the spectral calculus, we have
\begin{equation}\label{EQ:symbol-F}
\sigma_{F(E)}(j)=F(\lambda_{j}) I_{d_{j}}.
\end{equation}

\begin{proof}[Proof of Theorem  \ref{THM:inv}]
Once the basis $e_{j}^{k}$ is fixed,
the equivalence of (i), (ii) and (iv) follows from the equivalence of 
(A), (B) and (C) in Theorem \ref{THM:inv-rem}.

\smallskip
(ii) $\Longrightarrow$ (i).
We first note that both $ET$ and $TE$ are well-defined on $e_{j}^{k}$: for the former,
since $e_{j}^{k}$ is smooth, we have $Te_{j}^{k}\in L^{2}(M)$ and hence in $\Dcal'(M)$
where $E$ is well-defined as a pseudo-differential operator, while, for the latter, 
$E e_{j}^{k}=\lambda_{j} e_{j}^{k}\in H_{j}\subset C^{\infty}(M)$ and hence it is
in the domain of $T$.
The rest of the proof is identical to (E) $\Longrightarrow$ (A) in the proof of
Theorem \ref{THM:inv-rem2}.

\smallskip
(i) $\Longrightarrow$ (ii). This is the same as (A) $\Longrightarrow$ (E) in the proof of
Theorem \ref{THM:inv-rem2}.

\smallskip
(v) $\Longrightarrow$ (i). We 
take $f\in H_{j}.$ Then $P_{j} f=f\in C^{\infty}(M)$ so that by assumption (v)
we have
$$
Tf=TP_{j}f=P_{j}T f\in H_{j},
$$
implying (i).

\smallskip
(i) $\Longrightarrow$ (v).
We now assume in addition that $T$ is continuous on $\Dcal'(M)$.
First, we show that (i) implies that for any $g\in H_j^{\bot}\subset L^{2}(M)$, we have 
$\langle Tg,\overline{e_{j}^{k}}\rangle=0$
in the sense of distributions.
We can write 
\[
g=\sum\limits_{\ell\neq j}\sum\limits_{k=1}^{d_{\ell}}(g,e_{\ell}^k)e_{\ell}^k
\]
with the convergence in $L^{2}(M)$.
Hence
\[
Tg=\sum\limits_{\ell\neq j}\sum\limits_{k=1}^{d_{\ell}}(g,e_{\ell}^k)Te_{\ell}^k
\] 
with the convergence in $\Dcal'(M)$.
Since $Te_{\ell}^k\in H_{\ell}\subset H_j^{\bot} $ for $\ell\neq j$ we conclude that 
$Tg$ is orthogonal to $H_{j}$.

Let now $f\in C^{\infty}(M)$. Writing
$f=f_{1}+f_{2}$ with $f_{1}=P_{j}f$ so that 
$f_{1}\in H_{j}$ and $f_{2}\in H_{j}^{\bot}$ are necessarily smooth, and
$P_{j}f_{2}=0$, we have
$$
P_{j}Tf=P_{j}Tf_{1}+P_{j}Tf_{2}=Tf_{1}=TP_{j}f,
$$
since the above property implies that $P_{j}Tf_{2}=0$.

\smallskip
(vi) $\Longrightarrow$ (ii). Trivial.

\smallskip
(ii) $\Longrightarrow$ (vi).
For the following, we assume that $T$ is continuous on $\Dcal'(M)$.
Let us write $S:=E\circ T, D:=T\circ E$ and let $f\in L^2(M)$. We can write 
\[
f=\sum\limits_{j=0}^{\infty}\sum\limits_{k=1}^{d_j}(f,e_j^k)e_j^k
\]
with the series convergent in $L^{2}(M)$. 
Since both $S$ and $D$ are continuous on $\Dcal'(M)$, we now have
\[
Sf=\lim\limits_N\sum\limits_{j=0}^{N}\sum\limits_{k=1}^{d_j}(f,e_j^k)Se_j^k
=\lim\limits_N\sum\limits_{j=0}^{N}\sum\limits_{k=1}^{d_j}(f,e_j^k)De_j^k
=Df.
\]
The limit should be understood in $\Dcal'(M)$. Indeed, if we write 
\[
f_N=\sum\limits_{j=0}^{N}\sum\limits_{k=1}^{d_j}(f,e_j^k)e_j^k,
\]
then $f_N\rightarrow f$ in $L^2$ and hence also in $\Dcal'(M)$, 
which implies $S f_{N}\to Sf$ and $Df_{N}\to Df$ in $\Dcal'(M)$.
 \end{proof}

We now discuss how invariant operators can be expressed in terms of their symbols.

\begin{prop}\label{PROP:quant}
An invariant operator $T_{\sigma}$ associated to the symbol $\sigma$ can be written 
in the following way:
\begin{align}
T_{\sigma}f(x)=&\sum\limits_{\ell=0}^{\infty}\sum\limits_{m=1}^{d_{\ell}}(\sigma(\ell)\widehat{f}(\ell))_me_{\ell}^m(x)\label{form23}\\
=&\sum\limits_{\ell=0}^{\infty}[\sigma(\ell)\widehat{f}(\ell)]^{\top} e_{\ell}(x),\nonumber
\end{align}
where $[\sigma(\ell)\widehat{f}(\ell)]$ denotes the column-vector, and 
$[\sigma(\ell)\widehat{f}(\ell)]^{\top}e_{\ell}(x)$ denotes the multiplication
(the scalar product)
of the column-vector $[\sigma(\ell)\widehat{f}(\ell)]$ with the column-vector
$e_{\ell}(x)=(e_{\ell}^{1}(x),\cdots, e_{\ell}^{m}(x))^{\top}$.
In particular, we also have
\begin{equation}\label{EQ:Tsigma-e}
(T_{\sigma}e_{j}^{k})(x)=\sum\limits_{m=1}^{d_{j}} \sigma(j)_{mk}e_{j}^{m}(x).
\end{equation}
If $\sigma\in {\mathcal S}'(\Sigma)$ and $f\in C^{\infty}(M)$, the convergence in
\eqref{form23} is uniform.
\end{prop}
\begin{proof} 
Formula \eqref{form23} follows from Part (iv) of Theorem \ref{THM:inv},
with uniform convergence for $f\in C^{\infty}(M)$ in view of
\eqref{EQ:sym-tempe}.
Then, using \eqref{form23} and (\ref{eqdeltas}) we can calculate
\begin{align*}
(T_{\sigma}e_{j}^{k})(x)=&\sum\limits_{\ell=0}^{\infty}\sum\limits_{m=1}^{d_{\ell}}(\sigma(\ell)\widehat{e_{j}^{k}}(\ell))_me_{\ell}^m(x)\\ 
=&\sum\limits_{\ell=0}^{\infty}\sum\limits_{m=1}^{d_{\ell}}\left(\sum\limits_{i=1}^{d_{\ell}}(\sigma(\ell))_{mi}\widehat{e_{j}^{k}}(\ell,i)\right)e_{\ell}^m(x)\\
=&\sum\limits_{\ell=0}^{\infty}\sum\limits_{m=1}^{d_{\ell}}\sum\limits_{i=1}^{d_{\ell}}(\sigma(\ell))_{mi}\delta_{j\ell}\delta_{ki}e_{\ell}^m(x)\\
=&\sum\limits_{m=1}^{d_{j}}(\sigma(j))_{mk}e_{j}^m(x),
\end{align*}
yielding \eqref{EQ:Tsigma-e}.
\end{proof}

Theorem \ref{L2-abstract} characterising invariant operators bounded on 
$L^{2}(M)$ now becomes
\begin{thm}\label{L2} 
An invariant linear operator $T:\cinfm\rightarrow L^{2}(M)$ extends to a bounded
operator from $L^{2}(M)$ to $L^{2}(M)$ if and only if its symbol $\sigma$ satisfies
$$\sup\limits_{\ell\in\N_{0}}\|\sigma(\ell)\|_{op}<\infty,$$
where $ \|\sigma(\ell)\|_{op}=\|\sigma(\ell)\|_{{\mathscr L}(H_{\ell})}$ 
is the matrix multiplication
operator norm with respect to $H_{\ell}\simeq\ell^2(\ce^{d_{\ell}})$.
Moreover, we have
\[
\| T\|_{{\mathscr L}(L^{2}(M))} =\sup\limits_{\ell\in \N_{0}}\|\sigma(\ell)\|_{op}.
\]
\end{thm} 

This can be extended to Sobolev spaces. We will use 
the multiplication property for Fourier multipliers which is a direct
consequence of Proposition \ref{comp1-abstract}:

\begin{prop}\label{comp1} 
If $S,T:C^{\infty}(M)\to L^{2}(M)$ are invariant operators with respect to $E$ 
such that the domain of $S\circ T$ contains $C^{\infty}(M)$,
then $S\circ T:C^{\infty}(M)\to L^{2}(M)$ is also 
invariant with respect to $E$. Moreover, if $\sigma_S$ denotes the symbol of 
$S$ and $\sigma_T$ denotes the symbols of $T$ 
with respect to  the same orthonormal basis then 
\[\sigma_{S\circ T}=\sigma_S\sigma_T,\]
i.e. $\sigma_{S\circ T}(j)=\sigma_S(j)\sigma_T(j)$ for all $j\in\N_{0}.$
\end{prop}

Recalling Sobolev spaces $H^{s}(M)$ in \eqref{EQ:Sob-char} we have:

\begin{cor}\label{COR:Sobolev} 
Let an invariant linear operator $T:\cinfm\rightarrow C^{\infty}(M)$ 
have symbol $\sigma_{T}$ for which there exists $C>0$  and $m\in\mathbb R$ such that 
\[
\|\sigma_{T}(\ell)\|_{op}\leq C(1+\lambda_{\ell})^{\frac{m}{\nu}}
\]
holds for all $\ell\in\N_{0}$. Then $T$ extends to a bounded operator
from $H^{s}(M)$ to $H^{s-m}(M)$ for every $s\in\mathbb R$.
\end{cor} 
\begin{proof}
We note that by \eqref{EQ:Sob-char} the condition that 
$T:H^{s}(M)\to H^{s-m}(M)$ is bounded is equivalent to 
the condition that the operator
$$S:=(I+E)^{\frac{s-m}{\nu}}\circ T\circ (I+E)^{-\frac{s}{\nu}}$$ is bounded
on $L^{2}(M)$.
By Proposition \ref{comp1} and the fact that the powers of $E$ are
pseudo-differential operators with diagonal symbols, 
see \eqref{EQ:symbol-F}, we have
$$
\sigma_{S}(\ell)=(1+\lambda_{\ell})^{-\frac{m}{\nu}}\sigma_{T}(\ell).
$$
But then $\|\sigma_{S}(\ell)\|_{op}\leq C$ for all $\ell$ in view of the assumption on
$\sigma_{T}$, so that
the statement follows from Theorem \ref{L2}.
\end{proof}

\section{Schatten classes of operators on compact manifolds}
\label{SEC:Schatten-mfds}

In this section we give an application of the constructions in the previous
section to determine the membership of operators in Schatten classes
and then apply it to a particular family of operators on $L^{2}(M)$.

As a consequence of Theorem \ref{schchr-abstract}, 
we can now characterise invariant operators in Schatten classes on compact manifolds. 
We note that this characterisation does not assume any regularity of the kernel nor
of the symbol. Once we observe that the conditions for the membership
in the Schatten classes depend only on the basis $e_{j}^{k}$ and not on 
the operator $E$, we immediately obtain:

\begin{thm}\label{schchr} Let $0<r<\infty$. 
An invariant operator $T:L^2(M)\rightarrow L^2(M)$ is in $S_r(L^2(M))$ 
if and only if  $\sum\limits_{\ell=0}^{\infty}\|\sigma_{T}(\ell)\|_{S_r}^r<\infty$. 
Moreover
\[\|T\|_{S_r(L^2(M))}^r=\sum\limits_{\ell=0}^{\infty}\|\sigma_{T}(\ell)\|_{S_r}^r.\]
If an invariant operator $T:L^2(M)\rightarrow L^2(M)$ is in the trace class
$S_1(L^2(M))$, then 
\[\Tr(T)=\sum\limits_{\ell=0}^{\infty}\Tr(\sigma_{T}(\ell)).\]
\end{thm}

\begin{rem}\label{rema} In Section \ref{SEC:Lie-groups} 
we will establish a relation between the notion of symbol introduced 
in Theorem \ref{THM:inv} 
and the corresponding symbol in the setting of 
compact Lie groups (cf. \cite{rt:book, rt:groups}). 
In particular the characterisation above extends the one obtained in 
Theorem 3.7 of \cite{dr13:schatten}.
\end{rem}

We now apply Theorem \ref{schchr} to determining which powers of $E$ belong
to which Schatten classes. But first we record a useful relation between
the sequences $\lambda_{j}$ and $d_{j}$ of eigenvalues of $E$ and their
multiplicities.

\begin{prop}\label{weyl2a}
Let $M$ be a closed manifold of dimension $n$, and let $E\in\Psi_{+e}^{\nu}(M)$, with $\nu>0$. 
Then there exists a constant $C>0$ such that we have
\begin{equation}\label{EQ:Weyl}
d_j\leq C (1+\lambda_j)^{\frac{n}{\nu}}
\end{equation}
for all $j\geq 1$.
Moreover, we also have
\beq
\label{asymp2}
\sum\limits_{j=1}^{\infty}d_j(1+\lambda_j)^{-q}<\infty \quad{\textrm{ if and only if }}\quad q>\frac{n}{\nu}.
\eq
\end{prop}
\begin{proof}
Since $(1+\lambda_j)^{1/\nu}$ are the eigenvalues of the first-order elliptic positive
operator $(I+E)^{1/\nu}$ with multiplicities $d_j$, the Weyl eigenvalue
counting formula for the operator $(I+E)^{1/\nu}$ gives
$$
\sum_{j:\ (1+\lambda_j)^{1/\nu}\leq \lambda} d_j=C_0\lambda^n+O(\lambda^{n-1})
$$
as $\lambda\to\infty$. This implies $d_j\leq C(1+\lambda_j)^{n/\nu}$ for sufficiently large
$\lambda_j$, implying the estimate \eqref{EQ:Weyl}.

To prove \eqref{asymp2},
let us denote $T:=(I+E)^{-q/2}$. Then the eigenvalues of $T$ are $(1+\lambda_j)^{-q/2}$ with
multiplicities $d_j$. This implies
\begin{equation}\label{EQ:HSK}
\sum\limits_{j=0}^{\infty}d_j(1+\lambda_j)^{-q}=\|T\|^2_{S_{2}}\asymp \|K\|^2_{L^2(M\times M)}.
\end{equation}
By the functional calculus of pseudo-differential operators, we have
$T\in \Psi^{-\nu q/2}(M)$, and so its integral 
kernel $K(x,y)$ is smooth for $x\not=y$, and near the diagonal
$x=y$, identifying points with their local coordinates, we have
$$
|K(x,y)|\leq C_\alpha |x-y|^{-\alpha},
$$
for any $\alpha>n-\nu q/2$, 
see e.g. \cite{Duis:BK-FIO-2011} or \cite[Theorem 2.3.1]{rt:book}. Thus
order is sharp with respect to the order of the operator.
Therefore, $K\in L^2(M\times M)$ if and only if there exists $\alpha$ such that
$n>2\alpha>2n-\nu q$. Together with \eqref{EQ:HSK} this implies \eqref{asymp2}.
\end{proof}

\begin{prop}\label{PROP:elliptic}
Let $M$ be a closed manifold of dimension $n$, and let $E\in \Psi^{\nu}_{+e}(M)$ be a
positive elliptic pseudo-differential operator of order $\nu>0$.
Let $0<p<\infty$.
Then we have
\begin{equation}\label{EQ:ell-powers}
(I+E)^{-\frac{\alpha}{\nu}}\in S_{p}(L^{2}(M)) \;\textrm{ if and only if }\;
\alpha>\frac{n}{p}.
\end{equation}
\end{prop}
\begin{proof}
We note that the operator $(I+E)^{-\frac{\alpha}{\nu}}$ is positive definite, 
its singular values are $(1+\lambda_{j})^{-\frac{\alpha}{\nu}}$ with multiplicities 
$d_{j}$. Therefore, 
$$
\|(I+E)^{-\frac{\alpha}{\nu}}\|_{S_p}^p
= \sum\limits_{j=0}^{\infty}d_j (1+\lambda_j)^{-\frac{\alpha p}{\nu}},
$$
which is finite if and only if $\alpha p>n$ by (\ref{asymp2}),
implying the statement.
\end{proof}

\section{Relation to the setting of compact Lie groups}
\label{SEC:Lie-groups}

In the recent work \cite{dr13:schatten} the authors studied Schatten classes of operators on 
compact Lie groups. We now explore how the notion of the symbol from
Theorem \ref{THM:inv}
corresponds to the 
matrix-valued symbols on compact Lie groups, and how the results for Schatten classes
correspond to each other when $M=G$ is a compact Lie group. 
In this and the following sections we assume that all operators are continuous on $\Dcal'(G)$
so that the integral kernels of such operators are distributions.

We will give two types of decompositions of $L^{2}(G)$ into $H_{j}$'s as in
\eqref{EQ:sum}. First, we choose $H_{j}$'s determined by unitary irreducible
representations of $G$. However, in this case the condition
\eqref{EQ:lambdas} may fail. Consequently, to view this analysis 
as a special case of the construction on manifolds in Section \ref{SEC:invariant} with
condition \eqref{EQ:lambdas}, we group representations corresponding to the same
eigenvalue of the Laplacian together, to form a coarser decomposition of $L^{2}(G)$
into a direct sum of finite dimensional subspaces. The example of this types of
partitions is given in Remark \ref{REM:torus} in the case of the torus $\Tn$.

Now, we recall some basic definitions.
Let $G$ be a compact Lie group of dimension $n$ equipped
with the normalised Haar measure. 
Let $\widehat{G}$ denote the set of equivalence classes of continuous irreducible unitary 
representations of $G$. Since $G$ is compact, the set $\widehat{G}$ is discrete.  
For $[\xi]\in \widehat{G}$, by choosing a basis in the representation space of $\xi$, we can view 
$\xi$ as a matrix-valued function $\xi:G\rightarrow \ce^{d_{\xi}\times d_{\xi}}$, where 
$d_{\xi}$ is the dimension of the representation space of $\xi$. 
By the Peter-Weyl theorem the collection
$$
\left\{ \sqrt{d_\xi}\,\xi_{ij}: \; 1\leq i,j\leq d_\xi,\; [\xi]\in\Gh \right\}
$$
is the orthonormal basis of $L^2(G)$.
If $f\in L^1(G)$ we define its group Fourier transform at $\xi$ by 
 \beq\label{Fourier}
 \efee_{G} f(\xi)\equiv \widehat{f}(\xi):=\int_{G}f(x)\xi(x)^*dx,
 \eq  
where $dx$ is the normalised Haar measure on $G$. 
If $\xi$ is a matrix representation, we have $\widehat{f}(\xi)\in\ce^{d_{\xi}\times d_{\xi}} $. 
We note that this Fourier transform is different from the one we considered on manifolds
in \eqref{EQ:F-coeff} which produced vector-valued Fourier coefficients instead of the
matrix-valued ones obtained in \eqref{Fourier}.

The Fourier inversion formula is a consequence of the Peter-Weyl theorem, so that we have
\beq 
f(x)=\sum\limits_{[\xi]\in \widehat{G}}d_{\xi} \Tr(\xi(x)\widehat{f}(\xi)).
\eq

For each $[\xi]\in \widehat{G}$, the matrix elements of $\xi$ 
are the eigenfunctions for the Laplacian $\mathcal{L}_G$ 
(or the Casimir element of the universal enveloping algebra), 
with the same eigenvalues which we denote by 
$-\lambda^2_{[\xi]}$, so that we have
\begin{equation}\label{EQ:Lap-lambda}
-\mathcal{L}_G\xi_{ij}(x)=\lambda^2_{[\xi]}\xi_{ij}(x)\qquad\textrm{ for all } 1\leq i,j\leq d_{\xi}.
\end{equation} 
For a thorough discussion of Laplacians on compact Lie groups we refer to
\cite{Stein:BOOK-topics-Littlewood-Paley}.

The weight for measuring the decay or growth of Fourier coefficients in this setting is 
$\jp{\xi}:=(1+\lambda^2_{[\xi]})^{\half}$, the eigenvalues of the elliptic first-order pseudo-differential operator 
$(I-\mathcal{L}_G)^{\half}$.
The Parseval identity takes the form 
$$
\|f\|_{L^2(G)}= \left(\sum\limits_{[\xi]\in \widehat{G}}d_{\xi}\|\widehat{f}(\xi)\|^2_{\HS}\right)^{\half},\quad
\textrm{ where }\,\,
\|\widehat{f}(\xi)\|^2_{\HS}=\Tr(\widehat{f}(\xi)\widehat{f}(\xi)^*),$$
which defines the norm on 
$\ell^2(\widehat{G})$. 

For a linear continuous operator $A$ from $C^{\infty}(G)$ to $\mathcal{D}'(G) $ 
we define  its {matrix-valued symbol} $\tau_A(x,\xi)\in\cdxi$ by 
\begin{equation}\label{EQ:A-symbol}
\tau_A(x,\xi):=\xi(x)^*(A\xi)(x)\in\cdxi.
\end{equation}
Then one has (\cite{rt:book}, \cite{rt:groups}) the global quantization
\begin{equation}\label{EQ:A-quant}
Af(x)=\sum\limits_{[\xi]\in \widehat{G}}d_{\xi}\Tr(\xi(x)\tau_A(x,\xi)\widehat{f}(\xi))
\end{equation}
in the sense of distributions, and the sum is independent of the choice of a representation $\xi$ from each 
equivalence class 
$[\xi]\in \widehat{G}$. If $A$ is a linear continuous operator from $C^{\infty}(G)$ to $C^{\infty}(G)$,
the series \eqref{EQ:A-quant} is absolutely convergent and can be interpreted in the pointwise
sense. We will also write $A=\Op(\tau_A)$ for the operator $A$ given by
the formula \eqref{EQ:A-quant}.
We refer to \cite{rt:book, rt:groups} for the consistent development of this quantization
and the corresponding symbolic calculus.

In the case of a left-invariant operator $A$, its symbol $\tau_{A}$ is independent of $x$,
and formula \eqref{EQ:A-symbol} reduces to
\begin{equation}\label{EQ:A-symbol-inv}
\tau_A(\xi)=\xi(x)^*(A\xi)(x)=A\xi(e),
\end{equation}
where $e$ is the unit element of the group.

We can now establish a correspondence between the two frameworks, the one in this paper and the one given in \cite{dr13:schatten}. 
 In the setting of compact Lie groups the unitary dual being discrete, we can enumerate the representations as $\xi_j$ for $0\leq j<\infty .$
 The indices $(i,\ell)$ of each matrix $\xi(x)$ will be enumerated following the lexicographical order $((i,\ell)\leq (i',\ell ') \mbox{ if } i<i' \mbox{ or } (i=i'\mbox{ and } \ell\leq \ell '))$. In this way, we 
 fix the orthonormal basis $\{e_j^k\}$ given by 
\beq\label{ort11}
\{e_j^k\}_{1\leq k\leq d_j}=
\left\{\sqrt{d_{\xi _j}}(\xi _j)_{i \ell}\right\}_{1\leq i,\ell\leq d_{\xi_j}},\eq
where $d_j=d_{\xi_j}^2$ and $k$ represents an entry of the matrix of the 
representation following the lexicographical order described above. 
Then we have the subspaces
\begin{equation}\label{EQ:Hj-G}
H_{j}\equiv H_{[\xi_{j}]}:= {\rm span}\{ (\xi _j)_{i \ell}: \; 1\leq i,\ell\leq d_{\xi_j} \}.
\end{equation}

On a compact Lie group $G$ we can consider $E$ to be a bi-invariant Laplacian,
see Stein \cite{Stein:BOOK-topics-Littlewood-Paley} for a discussion of such operators.
Then, in view of the Peter-Weyl theorem,
the functions $\{e_j^k\}_{1\leq k\leq d_j}$ are its eigenfunctions, 
with norm one in $L^{2}(G)$ with respect to the normalised Haar measure, and
corresponding to the same eigenvalue $\lambda_{j}$.
However, the condition \eqref{EQ:lambdas} does not hold in general since 
non-equivalent representations in $\widehat{G}$ may give the same eigenvalues 
of the Laplacian.

We now observe that there is also a correspondence between the vector-valued Fourier transform introduced in \eqref{EQ:F-coeff} and the matrix-valued Fourier transform defined in (\ref{Fourier}). Such correspondence can be established by applying once more the lexicographical order to the matrix-valued Fourier transform (\ref{Fourier}). 

 In order to study such correspondence, for $d\in\ene$ we will define a bijection from the set of indices  of the matrix-symbol $\{1,\ldots,d\}^2$ onto the set of indices $\{1,\ldots,d^2\}$ and calculate its inverse. If $(j,k)\in\{1,\ldots,d\}^2$ we define
\[\Gamma_d(j,k):=(j-1)d+k.\]
The function $\Gamma_d$ is surjective, indeed if $t\in\{1,\ldots,d^2\}$, $j$ can be obtained from
\[j=\left\lfloor\frac{t-1}{d}\right\rfloor+1,\]
where $\lfloor \cdot \rfloor$ denotes the function defined for $x\geq 0$ by $\lfloor x \rfloor =\max\{y\in\ene_0:y\leq x\}$.

For the term $k$ we observe that
\[j-1=\left\lfloor\frac{t-1}{d}\right\rfloor,\]
hence 
\[k=t-\left\lfloor\frac{t-1}{d}\right\rfloor d.\]
Since we are dealing with finite sets with the same number of elements, the injectivity of $\Gamma$ follows.

We can now establish correspondences between the Fourier transforms on $G=M$, for $M$ 
viewed as both a compact manifold and a compact Lie group. 
Taking into account (\ref{Fourier}) and \eqref{ort11} we obtain  
\beq\label{FMG}
({\mathcal{F}}_Mf)(i,t)= (f,e_{i}^{t})_{L^{2}}=
\sqrt{d_{\xi_{i}}}(({\mathcal{F}}_Gf)(\xi _i))_{(t-\left\lfloor\frac{t-1}{d_{\xi_{i}}}
\right\rfloor d_{\xi_{i}}, \left\lfloor\frac{t-1}{d_{\xi_{i}}}\right\rfloor+1)},
\eq
for $i\in\ene_0$, $1\leq t \leq d_{i}=d_{\xi_i}^2$. In the another direction we have
\beq\label{FMGa} 
(({\mathcal{F}}_Gf)(\xi_{\ell}))_{i,j}=
\frac{1}{\sqrt{d_{\xi_{\ell}}}}({\mathcal{F}}_Mf)(\ell, \Gamma_{d_{\xi_{\ell}}}(j,i)),
\eq
for $1\leq i,j \leq d_{\xi_{\ell}}$.\\

 For the sake of simplicity, we introduce the following notation:
 
 \[
 \psi(t,d):= \left\lfloor\frac{t-1}{d}\right\rfloor+1, \,\,
 \phi(t,d):=t-\left\lfloor\frac{t-1}{d}\right\rfloor d,
 \] 
where $t\in\{1,\ldots,d^2\}$. With this notation formula (\ref{FMG}) becomes 
\beq\label{FMG2}
({\mathcal{F}}_Mf)(\ell,m)= \sqrt{d_{\xi_{\ell}}}(({\mathcal{F}}_Gf)(\xi _{\ell})_{(\phi(m,d_{\xi_{\ell}}), \psi(m,d_{\xi_{\ell}}))}.
\eq
We also have 
\[e_j^k=(\sqrt{d_{\xi_{j}}}\,\xi_{j})_{(\psi(k,d_{\xi_{j}}), \phi(k,d_{\xi_{j}}))}.\]
In the calculations below we will use the following basic relations for the Fourier transform on a compact Lie group $G$:
$$
({\mathcal{F}}_G(\eta_{rs})(\eta))_{ij}=\int\limits_{G}\eta _{rs}(x)\overline{\eta_{ji}(x)}dx
=\frac{1}{d_{\eta}}\delta_{(i,j),(s,r)},
$$
which means that ${\mathcal{F}}_G(\eta_{rs})(\eta)$ is the matrix of dimension $d_{\eta}\times d_{\eta}$ with the only entry different from zero 
equal to $\frac{1}{d_{\eta}}$ in the position $(s,r)$. We will denote this matrix by $\frac{1}{d_{\eta}}(\delta_{(i,j),(s,r)})_{ij},$
and we have also 
$\delta_{(i,j),(s,r)}=1$ if $i=s$ and $r=j$, and 
$\delta_{(i,j),(s,r)}=0$ if $i\not=s$ or $r\not=j$.

Thus, for an invariant operator we obtain 
\begin{align}
(\efeG (T(\xi_{rs})))(\xi)=\tau(\xi)({\mathcal{F}}_G(\xi_{rs})(\xi))
= \tau(\xi)\frac{1}{d_{\xi}}(\delta_{(i,j),(s,r)})_{ij}\label{ght}.
\end{align}
 
In other words $(\efeG (T(\xi_{rs})))(\xi)$ is a matrix of dimension $d_{\xi}\times d_{\xi}$ with all the columns zero except for
the $r$-column which is equal to the $s$-column of $\frac{1}{d_{\xi}}\tau(\xi)$.\\

We shall denote by $\sigma$ the symbol corresponding to $T$ and consider
the orthonormal basis $\{e_j^k\}$ defined in (\ref{ort11}) in the sense of (\ref{invadef}) on manifolds. 
The symbol introduced in (\ref{EQ:A-symbol}) in the sense of groups will be denoted by $\tau$. 
We now can find formulae relating the symbols $\tau$ and $\sigma$. We begin by finding a formula for $\sigma$ in terms of $\tau$. By   (\ref{FMG2}), (\ref{invadef}) and (\ref{ght}) we obtain
\begin{align*}\sigma(\ell)_{mi}=&(\efeM (Te_{\ell}^i))(\ell,m)\\
=&\sqrt{d_{\xi_{\ell}}}((\efeG (Te_{\ell}^i))(\xi _{\ell}))_{(\phi(m,d_{\xi_{\ell}}), \psi(m,d_{\xi_{\ell}}))}\\
=&\sqrt{d_{\xi_{\ell}}}((\efeG (T(\sqrt{d_{\xi_{\ell}}}\xi_{\ell})_{\psi(i,d_{\xi_{\ell}}),
\phi(i,d_{\xi_{\ell}})}))(\xi _{\ell}))_{(\phi(m,d_{\xi_{\ell}}), \psi(m,d_{\xi_{\ell}}))}\\
=&d_{\xi_{\ell}}((\efeG (T(\xi_{\ell})_{\psi(i,d_{\xi_{\ell}}),\phi(i,d_{\xi_{\ell}})}))
(\xi _{\ell}))_{(\phi(m,d_{\xi_{\ell}}), \psi(m,d_{\xi_{\ell}}))}\\
=&d_{\xi_{\ell}}d_{\xi_{\ell}}^{-1}(\tau(\xi_{\ell})(\delta _{((p,q),(\phi(i,d_{\xi_{\ell}}),\psi(i,d_{\xi_{\ell}}))})_{pq})_{(\phi(m,d_{\xi_{\ell}}), \psi(m,d_{\xi_{\ell}}))}\\
=&\tau(\xi_{\ell})_{(\phi(m,d_{\xi_{\ell}}),\phi(i,d_{\xi_{\ell}}))}
\delta_{\psi(i,d_{\xi_{\ell}}),\psi(m,d_{\xi_{\ell}})}.
\end{align*}

Therefore, we obtain
\beq\label{taus2} {\sigma(\ell)_{mi}= }\left\{
\begin{array}{rl}
\tau(\xi_{\ell})_{(\phi(m,d_{\xi_{\ell}}),\phi(i,d_{\xi_{\ell}}))}\,,&\, 
\mbox{ if }\psi(m,d_{\xi_{\ell}})=\psi(i,d_{\xi_{\ell}}),\\
0\,\,\,\,\,\,\,\,,&\, \mbox{ otherwise}.\\
\end{array} \right. \eq

We note that both functions $\phi$ and $\psi$ are periodic with respect to the first 
parameters $i$ and $m$, implying that there is a periodic structure in the `big' manifold-symbol
$\sigma$ composed of some copies of the `small' group-symbol $\tau$.

We will now give a graphical description of the relations (\ref{taus2}) between the two symbols. 
The entries of $\tau(\xi_{\ell})$ are distributed inside the matrix-symbol $\sigma$ according to (\ref{taus2}): setting $d:=d_{\xi_{\ell}}$ it is
\begingroup
    \fontsize{9pt}{12pt}
\[
  \begin{blockarray}{ccccccccccccccc}
	   &   &   &  &   &  &i & & & &  & & & & &&   \\
     &   &   &  & d_{\xi_{\ell}}  & d_{\xi_{\ell}}+1 & & & & & 
  &  & &  
     d_{\xi_{\ell}}^{2}  \\
		 &   &   &  & \downarrow  & \downarrow & & & & & \downarrow & & &  \downarrow  \\
  \begin{block}{c\Left{}{(\mkern1mu}cccccccccccccc<{\mkern1mu})}
     & \tau(\xi_{\ell})_{11} & \tau(\xi_{\ell})_{12} & \cdots & \tau(\xi_{\ell})_{1d} & 0 &0&\cdots &0& &0&0&\cdots &0 \\
		& \tau(\xi_{\ell})_{21} & \tau(\xi_{\ell})_{22} & \cdots & \tau(\xi_{\ell})_{2d} & 0 &0&\cdots &0&\cdots &0&0&\cdots &0 \\
		& \vdots &\vdots   &  & \vdots   & \vdots   &\vdots  & &\vdots  &\cdots &\vdots  &\vdots  & &\vdots  & \\
	  & \tau(\xi_{\ell})_{d1} & \tau(\xi_{\ell})_{d2} & \cdots & \tau(\xi_{\ell})_{dd} & 0 &0&\cdots &0&\cdots &0&0&\cdots &0\\
		&0  & 0 & \cdots & 0 &\tau(\xi_{\ell})_{11} &\tau(\xi_{\ell})_{12}&\cdots &\tau(\xi_{\ell})_{1d}& &0&0&\cdots &0\\
				& 0 &0& \cdots & 0 &\tau(\xi_{\ell})_{21} &\tau(\xi_{\ell})_{22} &\cdots &\tau(\xi_{\ell})_{2d}&\cdots &0&0&\cdots &0 \\
		& \vdots &\vdots   &  & \vdots   & \vdots   &\vdots  & &\vdots  &\cdots &\vdots  &\vdots  & &\vdots  \\
	  &0 &  0 & \cdots & 0 & \tau(\xi_{\ell})_{d1}&\tau(\xi_{\ell})_{d2} &\cdots &\tau(\xi_{\ell})_{dd}&\cdots &0&0&\cdots &0\\
		& \vdots &\vdots   &  & \vdots   & \vdots   &\vdots  & &\vdots  &\cdots &\vdots  &\vdots  & &\vdots  \\
		& \vdots &\vdots   &  & \vdots   & \vdots   &\vdots  & &\vdots  &\cdots &\vdots  &\vdots  & &\vdots   \\
	&0  & 0 & \cdots & 0 & 0 &0&\cdots &0& &\tau(\xi_{\ell})_{11}&\tau(\xi_{\ell})_{12}&\cdots &\tau(\xi_{\ell})_{1d}\\
		& 0 &0 & \cdots & 0& 0&0 &\cdots &0&\cdots &\tau(\xi_{\ell})_{21}&\tau(\xi_{\ell})_{22}&\cdots &\tau(\xi_{\ell})_{2d} \\
		& \vdots &\vdots   &  & \vdots   & \vdots   &\vdots  & &\vdots  &\cdots &\vdots  &\vdots  & &\vdots   \\
	  &0 & 0 & \cdots &  0& 0&0 &\cdots &0&\cdots &\tau(\xi_{\ell})_{d1} &\tau(\xi_{\ell})_{d2}&\cdots &\tau(\xi_{\ell})_{dd}		
  \end{blockarray}
\]
\endgroup
On the other hand, given the symbol $\sigma$, an application of equations (\ref{taus2}) for $1\leq m,i\leq d_{\xi_{\ell}}$ gives 
\begin{equation}\label{taus2-2}
\tau(\xi_{\ell})_{mi}=\sigma(\ell)_{mi}, \,\mbox{ for }\, 1\leq m,i\leq d_{\xi_{\ell}}.
\end{equation}

The proposition below shows that the Schatten quasi-norms $\|\cdot\|_{S_r}$ 
of the symbols $\tau$ and $\sigma$ are in agreement when $M=G$ is a compact Lie group. 
Thus, our results in Section \ref{SEC:Schatten-mfds} are an extension of those in 
\cite{dr13:schatten} concerning Schatten classes. 
In particular, Theorem \ref{schchr} extents Theorem 3.7 of 
\cite{dr13:schatten} as announced in Remark \ref{rema}.  
We recall that on a compact Lie group $G$ we take $E$ to be a bi-invariant Laplacian.

\begin{prop}\label{scheq1} 
Let $G$ be a compact Lie group. 
If a linear operator $T:C^{\infty}(G)\rightarrow L^{2}(G)$, continuous on $\Dcal'(G)$,
is left-invariant then it is also invariant relative to the family of $H_{j}$'s as in
\eqref{EQ:Hj-G}
in the sense of Theorem \ref{THM:inv-rem} (in fact, it is also strongly invariant).

Let $T:C^{\infty}(G)\rightarrow L^{2}(G)$ be a left-invariant operator, and let $\sigma$ be its
symbol in the sense of Theorem \ref{THM:inv-rem} and $\tau$ its symbol
in the sense of groups as in \eqref{EQ:A-symbol-inv}. Then these symbols are
related by formulae \eqref{taus2}--\eqref{taus2-2}.

Consequently, for a bounded left-invariant operator
$T:L^2(G)\rightarrow L^2(G)$, for every $0<r<\infty$ we have
\[\|\sigma(\ell)\|_{S_{r}}^r=d_{\xi_{\ell}}\|\tau(\xi_{\ell})\|_{S_{r}}^r,\]
and, therefore,
\[\sum\limits_{\ell}\|\sigma(\ell)\|_{S_{r}}^r=\sum\limits_{\ell}d_{\xi_{\ell}}\|\tau(\xi_{\ell})\|_{S_{r}}^r.\]
\end{prop}

\begin{proof} 
The invariance in the sense of groups as in \eqref{EQ:A-symbol-inv}
of the group-left-invariant operators follows from the relation \eqref{taus2}
between symbols and from the characterisation in Theorem \ref{THM:inv-rem}.

For the following statements, since for Schatten quasi-norms we have
\[\|B\|_{S_{r}}=\||B|\|_{S_{r}},\]
we can assume that $\sigma,\tau$ are symmetric, and hence they can be also assumed 
diagonal. On the other hand, using the relation between $\sigma$ and $\tau$ 
in \eqref{taus2} and \eqref{taus2-2},
and by looking at the  
diagonal elements of $\sigma$ in (\ref{taus2}), we obtain 
\begin{align*} \|\sigma(\ell)\|_{S_{r}}^r=\sum\limits_{m=1}^{d_{\xi_{\ell}}^2}|\sigma(\ell)_{mm}|^r
=d_{\xi_{\ell}}\sum\limits_{m=1}^{d_{\xi_{\ell}}} |\tau(\xi_{\ell})_{mm}|^r
=d_{\xi_{\ell}}\|\tau(\xi_{\ell})\|_{S_{r}}^r.
\end{align*}
Thus $\|\sigma(\ell)\|_{S_{r}}^r=d_{\xi_{\ell}}\|\tau(\xi_{\ell})\|_{S_{r}}^r$ and,
therefore,
$\sum\limits_{\ell}\|\sigma(\ell)\|_{S_{r}}^r=\sum\limits_{\ell}d_{\xi_{\ell}}\|\tau(\xi_{\ell})\|_{S_{r}}^r.$
\end{proof}

We finish this section by describing an adaptation of the above construction to put
it in the framework of manifolds as described in Theorem \ref{THM:inv}.
In the case of the torus $\Tn$ this is indicated in Remark \ref{REM:torus}.
Recalling the definition of $H_{[\xi]}$ in \eqref{EQ:Hj-G} for each
$[\xi]\in\Gh$, and the notation $\lambda_{[\xi]}$ for the eigenvalues as
in \eqref{EQ:Lap-lambda}, 
for the sequence $0=\lambda_{0}^{2}<\lambda_{1}^{2}<\lambda_{2}^{2}<\ldots$ of
eigenvalues of $-\mathcal{L}_G$ counted without multiplicities we set
\begin{equation}\label{EQ:Hj-big}
 \widetilde{H_{\ell}}:=\bigoplus_{\underset{\lambda_{[\xi]}=\lambda_{\ell}}{[\xi]\in\Gh}}
 H_{[\xi]}=
 \bigoplus_{\underset{\lambda_{[\xi]}=\lambda_{\ell}}{[\xi]\in\Gh}}
 {\rm span}\{ \xi_{ik}: \; 1\leq i,k\leq d_{\xi} \},\quad \ell\in\N_{0}.
\end{equation}

The family of $\widetilde{H_{\ell}}$'s is the collection of eigenspaces of the
elliptic differential operator ${\mathcal L}_{G}$ for which the condition
\eqref{EQ:lambdas} is satisfied. The symbols $\sigma$ and 
$\widetilde{\sigma}$ of an invariant operator $T$ with 
respect to the partitions $H_{j}$'s and $\widetilde{H_{\ell}}$'s, respectively, are
related by
\begin{equation}\label{EQ:symbols}
\widetilde{\sigma}(\ell)=\bigotimes_{\underset{\lambda_{[\xi_{j}]}=\lambda_{\ell}}{[\xi_{j}]\in\Gh}}
\sigma(j),
\end{equation}
with $\widetilde{\sigma}(\ell)\in\ce^{\widetilde{d_{\ell}}\times \widetilde{d_{\ell}}}$
and $$\widetilde{d_{\ell}}= 
\sum_{\underset{\lambda_{[\xi_{j}]}=\lambda_{\ell}}{[\xi_{j}]\in\Gh}} d_{j}=
\sum_{\underset{\lambda_{[\xi_{j}]}=\lambda_{\ell}}{[\xi_{j}]\in\Gh}} d_{\xi_{j}}^{2}.$$
Recalling the relation
\eqref{taus2} 
between the symbol $\sigma$ in the sense of Theorem \ref{THM:inv-rem} and 
the group symbol $\tau$ as in \eqref{EQ:A-symbol-inv}, given by
\begin{equation}\label{EQ:symb1}
\sigma(j)\equiv \sigma(\xi_{j})=
\begin{pmatrix}
\tau(\xi_{j}) & 0 & \cdots & 0 \\
0 & \tau(\xi_{j}) & \cdots & 0 \\
\vdots & \vdots & \cdots & \vdots \\
0 & 0 & \cdots & \tau(\xi_{j})
\end{pmatrix},
\end{equation}
the formula \eqref{EQ:symbols}
provides the further relation between the symbol $\widetilde{\sigma}$ in the sense
of manifolds (in Theorem \ref{THM:inv}) and the group symbol $\tau$. Therefore, if
$\lambda_{[\xi_{1}]}=\ldots = \lambda_{[\xi_{m}]}=\lambda_{\ell}$ for non-equivalent 
representations $[\xi_{1}],\ldots,[\xi_{m}]\in\Gh$,
we have
\begin{equation}\label{EQ:symb2}
\widetilde{\sigma}(\ell)=
\begin{pmatrix}
\sigma(\xi_{1}) & 0 & \cdots & 0 \\
0 & \sigma(\xi_{2}) & \cdots & 0 \\
\vdots & \vdots & \cdots & \vdots \\
0 & 0 & \cdots & \sigma(\xi_{m})
\end{pmatrix}.
\end{equation}
In particular, we obtain

\begin{cor}\label{COR:symbs} 
Let $G$ be a compact Lie group and let
$T:C^{\infty}(G)\rightarrow L^{2}(G)$ be a linear operator, continuous on $\Dcal'(G)$.
If $T$ is left-invariant then it is also invariant relative to the operator 
${\mathcal L}_{G}$ 
(in the sense of Theorem \ref{THM:inv}).
The corresponding symbols are related by formulae
\eqref{EQ:symbols}--\eqref{EQ:symb2}.
\end{cor}

\section{Kernels of invariant operators on compact manifolds}
\label{SEC:kernels}

In this section we describe invariant operators relative to $E$ in terms of their kernels. 
We first observe that if $T=T_{\sigma}$ is invariant 
with symbol $\sigma$, expanding Proposition \ref{PROP:quant}
we can write
\begin{align*}
T_{\sigma}f(x)=&\sum\limits_{\ell=0}^{\infty}\sum\limits_{m=1}^{d_{\ell}}(\sigma(\ell)\widehat{f}(\ell))_me_{\ell}^m(x)\\
=&\sum\limits_{\ell=0}^{\infty}\sum\limits_{m=1}^{d_{\ell}}\sum\limits_{k=1}^{d_{\ell}}\sigma(\ell)_{mk}\widehat{f}(\ell)_ke_{\ell}^m(x)\\
=&\sum\limits_{\ell=0}^{\infty}\sum\limits_{m=1}^{d_{\ell}}\sum\limits_{k=1}^{d_{\ell}}\sigma(\ell)_{mk}e_{\ell}^m(x)\int\limits_{M} f(y)\overline{e_{\ell}^k(y)}dy\\
=&\int\limits_{M}\left(\sum\limits_{\ell=0}^{\infty}\sum\limits_{m=1}^{d_{\ell}}\sum\limits_{k=1}^{d_{\ell}}\sigma(\ell)_{mk}e_{\ell}^m(x)\overline{e_{\ell}^k(y)}\right)f(y)dy.\\
\end{align*}
Hence, the integral kernel $K(x,y)$ of $T_{\sigma}$ is given by
\beq 
K(x,y)=\sum\limits_{\ell=0}^{\infty}\sum\limits_{m=1}^{d_{\ell}}\sum\limits_{k=1}^{d_{\ell}}\sigma(\ell)_{mk}e_{\ell}^m(x)\overline{e_{\ell}^k(y)}.
\label{kerK}\eq
On the other hand we note that
\[\{e^m_{\ell}\otimes \overline{e^{m'}_{\ell'}}\}_{\ell,\ell'\geq 0}^{1\leq m\leq d_{\ell},1\leq m'\leq d_{\ell'}}\]
 defines an orthonormal basis of $L^2(M\times M)$. If $T$ is Hilbert-Schmidt on $L^{2}(M)$, 
 not necessarily invariant, then its kernel $K$ is square integrable and we can write
 its decomposition in this basis as
\begin{equation}\label{EQ:K-f1}
K(x,y)=\sum\limits_{\ell=0}^{\infty}\sum\limits_{\ell'=0}^{\infty}\sum\limits_{m=1}^{d_{\ell}}\sum\limits_{m'=1}^{d_{\ell'}}((\efemo)K)(\ell,m,\ell',m')e_{\ell}^m(x)\overline{e_{\ell'}^{m'}(y)},
\end{equation}
where $((\efemo)K)(\ell,m,\ell',m')$ denotes the Fourier coefficients of K with respect to the basis 
$\{e^m_{\ell}\otimes \overline{e^{m'}_{\ell'}}\}$ given by
\begin{align*}
((\efemo)K)(\ell,m,\ell',m')=&
(K,e_{\ell}^m(x)\overline{e_{\ell'}^{m'}(y)})_{L^2(M\times M)}
\nonumber\\
=&
\int\limits_{M\times M}K(x,y)\overline{e_{\ell}^m(x)}e_{\ell'}^{m'}(y)dxdy.
\end{align*}
We observe from \eqref{kerK} and \eqref{EQ:K-f1}
that $T$ is invariant relative to $(E,\{e^m_{\ell}\}_{\ell\geq 0}^{1\leq m\leq d_{\ell}})$ if and only if 
\begin{equation}
{\displaystyle ((\efemo)K)(\ell,m,\ell',m') }= \left\{
\begin{array}{rl}
 0,\, &\ell\neq\ell',\\
\sigma(\ell)_{mm'},&\, \ell=\ell'.
\end{array} \right.
\label{ksymb}\end{equation}
For example, from (\ref{kerK}) we obtain
\begin{align*} (K,e_{\ell}^m(x)\overline{e_{\ell'}^{m'}(y)})_{L^2(M\times M)}=&\int\limits_{M\times M}\left(\sum\limits_{j=0}^{\infty}\sum\limits_{k=1}^{d_{j}}\sum\limits_{i=1}^{d_{j}}\sigma(j)_{ki}e_{j}^k(x)\overline{e_{j}^i(y)}\right)\overline{e_{\ell}^m(x)}e_{\ell'}^{m'}(y)dxdy\\
=&\sum\limits_{j=0}^{\infty}\sum\limits_{k=1}^{d_{j}}\sum\limits_{i=1}^{d_{j}}\sigma(j)_{ki}\int\limits_{M}e_{j}^k(x)\overline{e_{\ell}^m(x)}dx\int\limits_{M}e_{\ell'}^{m'}(y)\overline{e_{j}^i(y)}dy\\
=&\left\{
\begin{array}{rl}
 0,\, &\ell\neq\ell',\\
\sigma(\ell)_{mm'},&\, \ell=\ell'.
\end{array} \right.
\end{align*}

We now introduce some notation which will be useful in order to define a 
suitable setting to study the above Fourier coefficients and the relation 
between operator's kernel and symbol. Let  
\[
 \Sigma(M\times M) := \left\{\tilde{\sigma}=(\tilde{\sigma}(\ell,m,\ell',m'))_{0\leq \ell,\ell'<\infty}^{1\leq m\leq d_{\ell}, 1\leq m'\leq d_{\ell'} } :\tilde{\sigma}(\ell,m,\ell',m')=0 \mbox{ if }\ell\neq \ell'\right\},
\]
\[\mathcal{K}:=\{K\in\mathcal{D}'(M\times M): 
K {\mbox{ defines an invariant operator relative to }}\, E\}.\]
We now consider the mapping 
\[K\mapsto (\efemo)K\]
from $\mathcal{K}$ into $\Sigma(M\times M)$. 
We can identify the family of symbols $\Sigma(M\times M)$ with the matrices $\bigcup\limits_{\ell}\mathbb{C}^{d_{\ell}\times d_{\ell}}$ by letting
\[\tilde{\sigma}\equiv\sigma\]
such that $\sigma(\ell)_{mm'}=\tilde{\sigma}(\ell,m,\ell,m')$. 
In this way we also get the identification
$$
\Sigma(M\times M)\simeq \Sigma_{M} =\Sigma
$$
with $\Sigma$ from \eqref{EQ:Sigma}.

\medskip
If $1\leq p<\infty$ we define
\[\ell^{p}(\Sigma)=\{\sigma\in\Sigma:\sum\limits_{\ell=0}^{\infty}\|\sigma(\ell)\|_{S_{p}}^p<\infty\}.\]
On $\ell^{p}(\Sigma)$ we define the norm
\[\|\sigma\|_{\ell^{p}(\Sigma)}:=\left(\sum\limits_{\ell=0}^{\infty}\|\sigma(\ell)\|_{S_{p}}^p\right)^{\frac 1p},\,\, 1\leq p<\infty .\]
If $p=\infty$ we define
\[\ell^{\infty}(\Sigma)=\{\sigma\in\Sigma:\sup\limits_{\ell\in\N_{0}}\|\sigma(\ell)\|_{op}<\infty\},\]
and we endow $\ell^{\infty}(\Sigma)$ with the norm 
\[\|\sigma\|_{\ell^{\infty}(\Sigma)}:=\sup\limits_{\ell\in\N_{0}}\|\sigma(\ell)\|_{op}.\]

The integral operator with kernel $K$ will be sometimes denoted by $T_K$. 
We note that in terms of the norms $\ell^{p}(\Sigma)$, for invariant operators
Theorem \ref{L2} 
can be formulated as
\begin{equation}\label{EQ:thm-L2}
T\in {\mathscr L}(L^{2}(M)) \Longleftrightarrow \sigma_{T}\in \ell^{\infty}(\Sigma),
\end{equation}
and Theorem \ref{schchr} can be formulated as
\begin{equation}\label{EQ:thm-Sp}
T\in S_{p}(L^{2}(M)) \Longleftrightarrow \sigma_{T}\in \ell^{p}(\Sigma)
\end{equation}
for $0<p<\infty$.

For the formulation of the following theorem we will use  the mixed-norm $L^p$ spaces $L_x^{p_1}L_y^{p_2}$ on the manifold $M$ for $1\leq p_1,p_2\leq \infty$. A measurable function $K(x,y)$ is said to belong to $L_x^{p_1}L_y^{p_2}(M\times M)$ if $$\|\|K(x,y)\|_{L_y^{p_2}}\|_{L_x^{p_1}}<\infty.$$ 
On $L_x^{p_1}L_y^{p_2}(M\times M)$ we consider the norm $\|\cdot\|_{L_x^{p_1}L_y^{p_2}}:=\|\|\cdot\|_{L_y^{p_2}}\|_{L_x^{p_1}}.$ We also define 
\[
L^{(p_1,p_2)}(M\times M):=L_x^{p_1}L_y^{p_2}(M\times M)\cap L_y^{p_1}L_x^{p_2}(M\times M),
\]
endowed with norm 
\[
\|\cdot\|_{L^{(p_1,p_2)}}:=\max\{\|\cdot\|_{L_x^{p_1}L_y^{p_2}},\, \|\cdot\|_{L_y^{p_1}L_x^{p_2}} \}.
\] We note that in general $L^{(p_1,p_2)}\neq L^{(p_2,p_1)}.$ 
The basic properties of mixed-norm $L^p$ spaces for many variables were first studied by Benedek and Panzone in \cite{bp:mxlp}. In particular they proved a version of Stein's Interpolation of operators theorem and as a consequence the Riesz-Thorin theorem in that setting. A slight modification allows us to apply the Riesz-Thorin theorem when the operator $T$ acts from a mixed-norm $L^p$ space to an $\ell^{p}(\Sigma)$-space. 

\begin{thm}\label{ffb2} 
If $1\leq p\leq 2$ and $K\in \mathcal{K}\cap L^{(p',p)}$, then $(\efemo)K\in \ell^{p'}(\Sigma)$, where $\frac 1p +\frac {1}{p'}=1$. 
\end{thm}

\begin{proof}  If $p=2$ we have $p'=2$. From $$K\in \mathcal{K}\cap L_x^{2}L_y^{2}\cap L_y^{2}L_x^{2}=\mathcal{K}\cap L_{x,y}^{2}\subset L_{x,y}^{2}$$ we get a Hilbert-Schmidt operator $T_K$.  On the other hand, by Theorem \ref{schchr} with $r=2$, if $\sigma$ is the symbol of $T_K$ then $\sum\limits_{\ell}\|\sigma(\ell)\|_{S_2}^2<\infty$. Hence and by (\ref{ksymb}) we obtain $(\efemo)K\in  \ell^{2}(\Sigma)$.  

\smallskip
For $p=1$ we have $p'=\infty$. If $$K\in \mathcal{K}\cap L_x^{\infty}L_y^{1}\cap L_y^{\infty}L_x^{1},$$ 
by Schur's Lemma we get $T_K\in \mathscr{L}(L^r(M))$ for all $1\leq r\leq \infty.$ In particular $T_K\in \mathscr{L}(L^2(M))$
 and by Theorem \ref{L2} the symbol $\sigma$ of $T_K$ satisfies
\[\sup\limits_{\ell}\|\sigma(\ell)\|_{op}<\infty.\]
By (\ref{ksymb})  we have 
\[ \|(\efemo)K\|_{\ell^{\infty}(\Sigma)}=\sup\limits_{\ell}\|\sigma(\ell)\|_{op}.\]
Hence $(\efemo)K\in\ell^{\infty}(\Sigma)$. We have shown that
\[(\efemo):\mathcal{K}\cap L^{(2,2)}\longrightarrow \ell^{2}(\Sigma)\]
and
\[(\efemo):\mathcal{K}\cap L^{(\infty,1)}\longrightarrow \ell^{\infty}(\Sigma).\]
By the Riesz-Thorin interpolation theorem between $L^{(r,s)}$ and $\ell^{p}(\Sigma)$ spaces
(cf. \cite[Theorem 2]{bp:mxlp}) we obtain
\[(\efemo):\mathcal{K}\cap L^{(p_1,p_2)}\longrightarrow \ell^{q}(\Sigma),\]
with $\frac 1p_1=\frac{1-\theta}{2}+\frac{\theta}{\infty},\, \frac 1p_2=\frac{1-\theta}{2}+\frac{\theta}{1}, \frac{1}{q}=\frac{1-\theta}{2}+\frac{\theta}{\infty}$ for $0\leq\theta\leq 1.$ 
Hence
\[p_1=\frac{2}{1-\theta},\, p_2=\frac{2}{1+\theta},\, q=\frac{2}{1-\theta}.\]
We observe that if $p=\frac{2}{1+\theta}$ then $\theta=\frac{2-p}{p}$ and 
$\frac{2}{1-\theta}=\frac{p}{p-1}=p'$.
Thus
\[(\efemo):\mathcal{K}\cap L^{(p',p)}\longrightarrow \ell^{p'}(\Sigma),\]
completing the proof.
\end{proof}
The following corollary is an immediate consequence of Theorems \ref{ffb2} and \ref{schchr}, 
it furnishes a sufficient kernel condition for Schatten classes with index $p'\geq 2$. 
\begin{cor}\label{ffb1} If $1\leq p\leq 2$ and $K\in \mathcal{K}\cap L^{(p',p)}(M\times M)$, then $T_K\in S_{p'}(L^2(M))$. 
\end{cor}
We recall that sufficient conditions of the type above in terms of kernels are not allowed for $0<p'<2$ as a consequence of a Carleman's example. Corollary \ref{ffb1} is known for general integral operators (cf. \cite[Theorem 3]{russo:lpg}). Here we have deduced a particular version for invariant operators with a simple proof by applying the notion of symbol.

We now describe another representation of the kernel as the 
`generalised' Fourier transform of the symbol. From formula (\ref{kerK}) we have 
\begin{align*} 
K(x,y)=&\sum\limits_{\ell=0}^{\infty}\sum\limits_{m=1}^{d_{\ell}}\sum\limits_{k=1}^{d_{\ell}}\sigma(\ell)_{mk}e_{\ell}^m(x)\overline{e_{\ell}^k(y)}\\
=&\sum\limits_{\ell=0}^{\infty}\Tr(e_{\ell}(x)^{\top}\sigma(\ell)\overline{e_{\ell}(y)})\\
=&\sum\limits_{\ell=0}^{\infty}\Tr(\sigma(\ell)\overline{e_{\ell}(y)}e_{\ell}(x)^{\top})\\
=&\sum\limits_{\ell=0}^{\infty}\Tr(\sigma(\ell)Q_{\ell}(x,y)),
\end{align*}
where $$Q_{\ell}(x,y)=\overline{e_{\ell}(y)}e_{\ell}(x)^{\top}\in\ce^{d_{\ell}\times d_{\ell}}.$$ 

We notice that the matrix-valued function $$(Q_{\ell}(x,y))_{mk}=\overline{e^m_{\ell}(x)}e^k_{\ell}(y)$$ is of rank one for every $\ell$. Indeed, $(Q_{\ell}(x,y))_{mk}$ is nothing else but the tensor product of the vectors $e_{\ell}(x), \overline{e_{\ell}(y)}\in\ce^{d_{\ell}}$. Since on a normed space $F$ we have $\|u\otimes v\|_{op}=\|u\|_{F}\|v\|_{F}$, we get
\[
\|Q_{\ell}(x,y)\|_{op}=\|e_{\ell}(x)\|_{\ell^2(\ce^{d_{\ell}})}\|e_{\ell}(y)\|_{\ell^2(\ce^{d_{\ell}})}.
\]
From (\ref{EQ:K-f1}) we have
\[\sigma(\ell)=\int\limits_{M\times M}K(x,y)Q_{\ell}(x,y)^*dxdy.\]
Hence
\begin{align*}\|\sigma(\ell)\|_{op}&\leq \|K\|_{L^1(M\times M)}\sup\limits_{x,y}\|Q_{\ell}(x,y)^*\|_{op}\\
&= \|K\|_{L^1(M\times M)}\sup\limits_{x,y}\|e_{\ell}(x)\|_{\ell^2(\ce^{d_{\ell}})}\|e_{\ell}(y)\|_{\ell^2(\ce^{d_{\ell}})}.
\end{align*}

\begin{rem} We point out that the mere condition $K\in L^1(M\times M)$ does not guarantee the $L^2$ boundedness of the corresponding integral operator $T$. 
 Indeed, consider $M=\T$,
 $g\in L^1(\T)\backslash L^2(\T), h\equiv 1\in L^1(\T)$,
 and the kernel $$K(\theta, \phi):=g(\theta)h(\phi)\in L^1(\T\times \T).$$ It is easy to see that 
 the kernel $K(\theta, \phi)$ does not define an operator from $L^2(\T)$ into
$L^2(\T)$. For example, with $f=1\in L^2(\T)$ we have
\[(T1)(\theta)=g(\theta)\int_{\T}h(\phi)d\phi =g(\theta)\notin L^2(\T).\]
\end{rem}

\section{Applications to the nuclearity of operators in $L^{p}(M)$}
\label{SEC:nuclearity}

We now turn to the study of nuclearity in $L^{p}$-spaces on closed manifolds. 
Sufficient conditions for $r$-nuclearity on 
$L^p$ on compact Lie groups have been established in \cite{dr13a:nuclp}.
The study of nuclearity on $L^p$ in this section
relies on the analysis of suitable kernel decompositions and the relation 
between kernels and symbols described in Section \ref{SEC:kernels}. 

Let $E$ and $F$ be two Banach spaces and $0<r\leq 1$, a linear operator $T$
from $E$ into $F$ is called {\em r-nuclear} if there exist sequences
$(x_{n}')\mbox{ in } E' $ and $(y_n) \mbox{ in } F$ so that
\beq Tx= \sum\limits_n \left <x,x_{n}'\right>y_n \,\mbox{ and }\,
\sum\limits_n \|x_{n}'\|^{r}_{E'}\|y_n\|^{r}_{F} < \infty.\label{rn}\eq
When $r=1$ they are known as {\em nuclear operators}, in that case this definition agrees with the concept of trace class operator in the setting of Hilbert spaces ($E=F=H$). 
More generally, Oloff proved in \cite{Oloff:pnorm} that the class of $r$-nuclear
operators coincides with the Schatten class $S_{r}(H)$ when $E=F=H$ and 
$0<r\leq 1$.

The concept of $r$-nuclearity was introduced by Grothendieck \cite{gro:me},
and it has application to questions of the distribution of eigenvalues of operators
in Banach spaces via e.g. the Grothendieck-Lidskii formula. 
We refer to \cite{dr13a:nuclp} for several conclusions 
in the setting of compact Lie groups
concerning
summability and distribution of eigenvalues of operators on $L^{p}$-spaces once
we have information on their $r$-nuclearity.
Since these arguments are then purely functional analytic, they apply equally well
in the present setting of closed manifolds; we omit the repetition but refer the reader
to \cite{dr13a:nuclp} for several relevant applications.

The $r$-nuclear operators on Lebesgue spaces are characterised by the following theorem 
(cf. \cite{del:tracetop}). 
In the statement below we  consider 
$({\Omega}_1,{\mathcal{M}}_1,\mu_1 )$ and $({\Omega}_2,{\mathcal{M}}_2,{\mu}_2)$ to be two $\sigma$-finite measure spaces. 
\begin{thm}\label{ch2} 
Let $1\leq p_1,p_2 <\infty$, $\,0<r\leq 1$ and let $q_1$ be such that
$\frac{1}{p_1}+\frac{1}{q_1}=1$. 
 An operator $T:L^{p_1}({\mu}_1)\rightarrow L^{p_2}({\mu}_2)$ is $r$-nuclear if and only if 
 there exist sequences
 $(g_n)_n$ in $L^{p_2}({\mu}_2)$, and $(h_n)_n$ in $L^{q_1}(\mu_1)$ such that
  $\sum \limits_{n=1}^\infty \| g_n\|_{L^{p_2}}^r
 \|h_n\|_{L^{q_1}}^r<\infty$, and such that for all $f\in L^{p_1}(\mu_1)$ we have
$$Tf(x)=\int\left(\sum\limits_{n=1}^{\infty}
  g_n(x)h_n(y)\right)f(y)d\mu_1(y), \quad \mbox{for  a.e }x.$$
\end{thm}
In order to study nuclearity on $L^p(M)$ spaces for a given compact manifold $M$ of dimension $n$, we introduce a function $\Lambda(j,k;n,p)$ which controls the $L^p$-norms of the family of eigenfunctions $\{e_{j}^{k}\}$ of the operator $E$, i.e. we will suppose that 
$\Lambda(j,k;n,p)$ is such that we have the estimates
\beq\|e_{j}^{k}\|_{L^{p}(M)}\leq \Lambda(j,k;n,p).\label{control}\eq

In particular, if $\Lambda$ is such a function we observe that 
\[\|e_{j}^{k}\|_{L^{p}(M)}\leq vol(M)^{\frac 1p}\Lambda(j,k;n,\infty).\]

When $M=G$ is a compact Lie group efficient
$\|e_{j}^{k}\|_{L^{p}(G)}$ bounds can be obtained (cf. \cite{dr13a:nuclp}). 
The estimation of $L^p$ norms for eigenfunctions of differential elliptic 
operators on general closed manifolds has been largely studied, see for instance \cite{sogge-z:ei}. 
Some examples will be given at the end of this section.
An example can be also obtained from the following simple lemma:

\begin{lem}\label{control1} 
Let $f$ be such that $\|f\|_{L^2(M)}=1$, then
\begin{itemize}
\item[(i)] $\|f\|_{L^p(M)}\leq (vol(M))^{\frac{2-p}{2p}}$ if $\,1\leq p\leq 2 .$
\item[(ii)] $\|f\|_{L^p(M)}\leq \|f\|_{L^{\infty}(M)}^{\frac{p-2}{p}}$ if $\,2\leq p<\infty .$
\end{itemize}
\end{lem}
\begin{proof} (i) By H\"older inequality we have
\begin{align*} \int_M|f(x)|^pdx\leq \left(\int_M|f(x)|^{p\frac 2p}dx\right)^{\frac p2}
\left(\int_M|1|^{p\frac{2}{2-p}}dx\right)^{\frac {2-p}{2}}
=(vol(M))^{\frac {2-p}{2}}.
\end{align*}
(ii) We also have
\begin{align*} \int_M|f(x)|^pdx=\int_M|f(x)|^{p-2}|f(x)|^{2}dx
\leq \|f\|_{L^{\infty}(M)}^{p-2},
\end{align*}
completing the proof. 
\end{proof}
We now formulate a sufficient condition for the $r$-nuclearity on $L^p(M)$ spaces 
as an application of the notion of the matrix-symbol on closed manifolds. 
Inspired by Lemma \ref{control1}, we will use the following function 
$\tilde{p}$ for $1\leq p\leq \infty$: 
\beq\label{ppp} \tilde{p}:=\left\{
\begin{array}{rl}
0\,,&\, \mbox{ if } 1\leq p\leq 2,\\
\frac{p-2}{p},\,&\,\mbox{ if } 2<p<\infty ,\\
1,\,&\,\mbox{ if } p=\infty .
\end{array} \right. \eq
For $p_{1}, p_{2}$ we denote their dual indices by $q_{1}:=p_{1}^{\prime}$,
$q_{2}:=p_{2}^{\prime}$.

\begin{thm}\label{main126} 
Let $1\leq p_1,p_2 <\infty$ and $0<r\leq 1$. Let $T:L^{p_1}(M)\to L^{p_2}(M)$ be a
strongly invariant linear 
continuous operator. Assume that its matrix-valued symbol $\sigma(\ell)$ satisfies 
\[\sum\limits_{\ell=0}^{\infty} \sum\limits_{m,k=1}^{d_{\ell}}
|\sigma(\ell)_{mk}|^r\Lambda(\ell,m;n,\infty)^{\tilde{p_2}r}\Lambda(\ell,k;n,\infty)^{\tilde{q_1}r}<\infty.\]
Then the operator $T:L^{p_1}(M)\rightarrow L^{p_2}(M)$ is $r$-nuclear. 
\end{thm}

\begin{proof} By \eqref{kerK} the kernel of $T$ is given by
\[K(x,y)=\sum\limits_{\ell=0}^{\infty}\sum\limits_{m=1}^{d_{\ell}}\sum\limits_{k=1}^{d_{\ell}}\sigma(\ell)_{mk}e_{\ell}^m(x)\overline{e_{\ell}^k(y)}.\]
We set
\[g_{\ell,m,k}(x):=\sigma(\ell)_{mk}e_{\ell}^m(x),\,h_{\ell,k}(y):=\overline{e_{\ell}^k(y)}.\]
Now, by Lemma \ref{control1} we have
\[
\|e_{\ell}^m\|_{L^{p}}\leq C_{p}\Lambda(\ell,m;n,\infty)^{\tilde{p}},
\]
 where $C_p=\max\{(vol(M))^{\frac{2-p}{2p}},1\}$.
We now observe that
\begin{align*} \sum\limits_{\ell,m,k}\|g_{\ell,m,k}\|_{L^{p_2}}^r\|h_{\ell,k}\|_{L^{q_1}}^r=&\sum\limits_{\ell=0}^{\infty}\sum\limits_{m,k=1}^{d_{\ell}}\|\sigma(\ell)_{mk}e_{\ell}^m\|_{L^{p_2}}^r\|\overline{e_{\ell}^k}\|_{L^{q_1}}^r\\
=&\sum\limits_{\ell=0}^{\infty}\sum\limits_{m,k=1}^{d_{\ell}}|\sigma(\ell)_{mk}|^r\|e_{\ell}^m\|_{L^{p_2}}^r\|e_{\ell}^k\|_{L^{q_1}}^r\\
\leq & (C_{p_{2}} C_{q_{1}})^{r}
\sum\limits_{\ell=0}^{\infty}\sum\limits_{m,k=1}^{d_{\ell}}|\sigma(\ell)_{mk}|^r\Lambda(\ell,m;n,\infty)^{\tilde{p_2}r}\Lambda(\ell,k;n,\infty)^{\tilde{q_1}r},
\end{align*}
finishing the proof in view of Theorem \ref{ch2}.
\end{proof}
In particular for formally self-adjoint invariant operators we can diagonalise each 
matrix $\sigma(\ell)$, so that we have
\begin{cor}\label{main12a} 
Let $1\leq p_1,p_2 <\infty$ and $0<r\leq 1$. 
Let $T:L^{p_1}(M)\to L^{p_2}(M)$ be a strongly invariant formally self-adjoint  
continuous operator. Assume that its matrix-valued symbol $\sigma(\ell)$ satisfies 
\[\sum\limits_{\ell=0}^{\infty}\sum\limits_{m=1}^{d_{\ell}}|\sigma(\ell)_{mm}|^r\Lambda(\ell,m;n,\infty)^{(\tilde{p_2}+\tilde{q_1})r}<\infty.\]
 Then the operator $T:L^{p_1}(M)\rightarrow L^{p_2}(M)$ is $r$-nuclear. 
\end{cor}

In some cases it is possible to simplify the sufficient condition above when the control function $\Lambda(\ell,m;n,\infty)$ is independent of $m$. For instance a classical result
(local Weyl law) due to H{\"o}rmander 
(\cite[Theorem 5.1]{ho:eigen1},  \cite[Chapter XXIX]{ho:apde4}) implies the following
estimate:
\begin{lem} \label{LEM:Horm}
Let $M$ be a closed manifold of dimension $n$. Let $E\in \Psi_{+e}^{\nu}(M)$, then
\beq\|e_{\ell}^m\|_{L^{\infty}}\leq C\lambda_{\ell}^{\frac{n-1}{2\nu}}.\label{asein}\eq 
\end{lem}

\begin{proof}
In order to explain this estimate we first consider the family of eigenvalues $\{\lambda_{\ell}\}$ of $E$ ordered in the increasing order
\[0=\lambda_{0}\leq \lambda_{1}\leq\cdots\lambda_{\ell}\leq\cdots\]  
and counted with multiplicity. For the projection $P_{\ell}(f)$ onto $H_{\ell}$, consider $E_{\lambda}f:=\sum\limits_{\lambda_{\ell}\leq\lambda}P_{\ell}(f)$ the associated partial sum operators. Its kernel is given by
\[E_{\lambda}(x,y)=\sum\limits_{\lambda_{\ell}\leq\lambda}\sum\limits_{m=1}^{d_{\ell}}e_{\ell}^m(x)\overline{e_{\ell}^m(y)}.\]
If $p(x,\xi)$ is the principal symbol of $E$, by Theorem 5.1 of  \cite{ho:eigen1} we have
\beq\label{ho23}E_{\lambda}(x,x)=\sum\limits_{\lambda_{\ell}\leq\lambda}\sum\limits_{m=1}^{d_{\ell}}|e_{\ell}^m(x)|^2=(2\pi)^{-n}\int\limits_{p(x,\xi)\leq\lambda}d\xi+R(x,\lambda)\eq
with \[|R(x,\lambda)|\leq C\lambda^{\frac{n-1}{\nu}},\, x\in M.\]

Since $E_{\mu}(x,x)$ is increasing right-continuous with respect to  $\mu$, the fact that the spectrum of $E$ is discrete, by the continuity of $\int\limits_{p(x,\xi)\leq\mu}d\xi$ with respect to $\mu$ and by taking left-hand limit in (\ref{ho23}) we obtain
\begin{align*}\lim\limits_{\mu\rightarrow \lambda^{-}}E_{\mu}(x,x)=\sum\limits_{\lambda_{\ell}<\lambda}\sum\limits_{m=1}^{d_{\ell}}|e_{\ell}^m(x)|^2
=(2\pi)^{-n}\int\limits_{p(x,\xi)\leq\lambda}d\xi+R(x,\lambda^{-}).
\end{align*} 
Hence
\[E_{\lambda_{\ell}}(x,x)-E_{\lambda_{\ell}^{-}}(x,x)=\sum\limits_{m=1}^{d_{\ell}}|e_{\ell}^m(x)|^2=R(x,\lambda_{\ell})-R(x,\lambda_{\ell}^{-}).\]
In particular, we have
\[|e_{\ell}^m(x)|\leq 2(\sqrt{R(x,\lambda_{\ell})}+\sqrt{R(x,\lambda_{\ell}^{-})}\,)\leq 2C\lambda_{\ell}^{\frac{n-1}{2\nu}},\]
which proves Lemma \ref{LEM:Horm}.
 \end{proof}

Thus $\Lambda(\ell;n,\infty)=C\lambda_{\ell}^{\frac{n-1}{2\nu}}$ furnishes an example of $\Lambda$ independent of $m$. For controls of type $\Lambda(\ell;n,\infty)$ we have a basis-independent condition: 

\begin{cor}\label{main12ab} 
Let $1\leq p_1,p_2 <\infty$ and $0<r\leq 1$. 
Let $T:L^{p_1}(M)\to L^{p_2}(M)$ be a strongly invariant formally self-adjoint  
continuous operator. Assume that its matrix-valued symbol $\sigma(\ell)$ satisfies 
\[
\sum\limits_{\ell=0}^{\infty}
\|\sigma(\ell)\|_{S_r}^r\Lambda(\ell;n,\infty)^{(\tilde{p_2}+\tilde{q_1})r}<\infty.
\]
Then the operator $T:L^{p_1}(M)\rightarrow L^{p_2}(M)$ is $r$-nuclear. 
In particular, if its matrix-valued symbol $\sigma(\ell)$ satisfies 
\begin{equation}\label{EQ:cor-nuc}
\sum\limits_{\ell=0}^{\infty}
\|\sigma(\ell)\|_{S_r}^r\lambda_{\ell}^{{\frac{(n-1)}{2\nu}}(\tilde{p_2}+\tilde{q_1})r}<\infty,
\end{equation}
then the operator $T:L^{p_1}(M)\rightarrow L^{p_2}(M)$ is $r$-nuclear. 
\end{cor}
\begin{proof} Since $T$ is $E$-invariant and formally self-adjoint, each matrix $\sigma(\ell) $ can be assumed diagonal, and the result follows from Corollary \ref{main12a} since 
\[
\sum\limits_{m=1}^{d_{\ell}}|\sigma(\ell)_{mm}|^r=\Tr(|\sigma(\ell)|^r)=\|\sigma(\ell)\|_{S_r}^r,
\]
completing the proof. The $r$-nuclearity under condition \eqref{EQ:cor-nuc}
follows by using Lemma \ref{LEM:Horm}
and
taking $\Lambda(\ell;n,\infty)=C\lambda_{\ell}^{\frac{n-1}{2\nu}}$.
\end{proof}

\begin{rem}\label{eqsmg}
If $M$ is a compact Lie group Corollary \ref{main12ab}  
absorbs Theorem 3.4 in \cite{dr13a:nuclp} by taking $E$ to be the Laplacian
and the family of eigenfunctions $\{e_{\ell}^{k}\}$ 
as in (\ref{ort11}). Indeed, since $|d_{\xi_{\ell}}^{\half}(\xi_{\ell})_{ij}(x)|\leq d_{\xi_{\ell}}^{\half} $ one can choose $\Lambda(\ell;\infty)=d_{\xi_{\ell}}^{\half}$ and taking into account that, by Lemma \ref{scheq1}, we have
\[\|\sigma(\ell)\|_{S_{r}}^r=d_{\xi_{\ell}}\|\tau(\xi_{\ell})\|_{S_{r}}^r,\]
we obtain
\[\sum\limits_{\ell}\|\sigma(\ell)\|_{S_r}^r\Lambda(\ell;\infty)^{(\tilde{p_2}+\tilde{q_1})r}=\sum\limits_{\ell}d_{\xi_{\ell}}^{1+{\half}{(\tilde{p_2}+\tilde{q_1})r}}\|\tau(\xi_{\ell})\|_{S_{r}}^r,\]
with a right-hand side equivalent to the term giving the sufficient condition in Theorem 3.4 of \cite{dr13a:nuclp}. Indeed,
\begin{align*} \half(\tilde{p_2}+\tilde{q_1})=&\half\left(1-\frac{2}{\max\{2,p_2\}}+1-\frac{2}{\max\{2,q_1\}}\right)\\
=&1-\frac{1}{\max\{2,q_1\}}-\frac{1}{\max\{2,p_2\}}\\
=&\frac{1}{\min\{2,p_1\}}-\frac{1}{\max\{2,p_2\}},
\end{align*}
which was the order obtained in \cite[Theorem 3.4]{dr13a:nuclp}
on compact Lie groups.
\end{rem}

In order to give another example we recall
Proposition \ref{weyl2a} with  
useful relations between the eigenvalues $\lambda_j$ and their multiplicities $d_j$. 
As a consequence of Corollary \ref{main12ab}  and Proposition \ref{weyl2a}, 
for the negative powers of the operator $E$ itself we obtain:

\begin{cor}\label{main12a3} 
Let $1\leq p_1,p_2 <\infty$ and $0<r\leq 1$. 
Let $E\in\Psi_{+e}^{\nu}(M) $. 
If $$\alpha>\frac{n}{r}+(\tilde{p_2}+\tilde{q_1})\frac{n-1}{2}$$
then the operator $(I+E)^{-\frac{\alpha}{\nu}}:L^{p_1}(M)\rightarrow L^{p_2}(M)$ is $r$-nuclear. 
\end{cor}

Note that if $p_{1}=p_{2}=2$, we have $\tilde{p_{2}}=\tilde{q_{1}}=0$, and since
Schatten class $S_{r}$ and $r$-nuclear class coincide on $L^{2}(M)$,
Proposition \ref{PROP:elliptic} shows that the statement of Corollary \ref{main12a3} 
is sharp in this case of indices. However, it does depend on the bounds for
eigenvalues which can be improved in the presence of additional structures as
discussed in Remark \ref{REM:bounds}.

\begin{proof}[Proof of Corollary \ref{main12a3}]
If we denote by $\lambda_{\ell}$ the eigenvalues of $E$, for $\alpha>0$ we observe that 
$\sigma_{(I+E)^{-\frac{\alpha}{\nu}}}(\ell)=(1+\lambda_{\ell})^{-\frac{\alpha}{\nu}}I_{d_{\ell}}$. 
Then 
\[
\|\sigma_{(I+E)^{-\frac{\alpha}{\nu}}}(\ell)\|_{S_r}^r=(1+\lambda_{\ell})^{-\frac{\alpha r}{\nu}}d_{\ell}.
\] 
Now by applying Corollary \ref{main12ab} we obtain
\begin{multline*}
\sum\limits_{\ell}\|\sigma(\ell)\|_{S_r}^r\lambda_{\ell}^{{\frac{(n-1)}{2\nu}}(\tilde{p_2}+\tilde{q_1})r}
\leq
C\sum\limits_{\ell}d_{\ell}(1+\lambda_{\ell})^{-\frac{\alpha r}{\nu}}(1+\lambda_{\ell})^{(\tilde{p_2}+\tilde{q_1})\frac{(n-1)r}{2\nu}}\\
=
C\sum\limits_{\ell}d_{\ell}(1+\lambda_{\ell})^{(-\alpha+(\tilde{p_2}+\tilde{q_1})\frac{(n-1)}{2})
\frac{r}{\nu}}
<\infty,
\end{multline*}
if $q=(\alpha-(\tilde{p_2}+\tilde{q_1})\frac{(n-1)}{2})\frac{r}{\nu}>\frac{n}{\nu}$ by Proposition \ref{weyl2a}. But this is equivalent to the condition 
$\alpha>\frac{n}{r}+(\tilde{p_2}+\tilde{q_1})\frac{n-1}{2}$.
\end{proof}

\begin{rem}\label{REM:bounds}
As we pointed out in Remark \ref{eqsmg}, on compact Lie groups we can always choose 
$E$ to be a Laplacian with an
orthonormal basis
given by rescaled matrix elements of representations, for which we can take
$\Lambda(\ell;\infty)=d_{\xi_{\ell}}^{\half}=d_{\ell}^{\frac14}$. 
At the same time, if $E$ is an operator of second order
(so that $\nu=2$) the best we can hope
for on closed manifolds in general is 
$\Lambda(\ell;n,\infty)=C\lambda_{\ell}^{\frac{n-1}{4}}$ given by Lemma \ref{LEM:Horm}.
In view of \eqref{EQ:Weyl}, we always have 
$d_\ell^{\frac14}\lesssim \lambda_{\ell}^{\frac{n}{8}}$, so that this choice on compact
Lie groups is
better than the general bound $\Lambda(\ell;n,\infty)=C\lambda_{\ell}^{\frac{n-1}{4}}$ above. 
Partly, this is explained by the presence of the additional (group) structure in this
case. The other point is that there is a difference in finding $L^{\infty}$-estimates for 
elements of {\em any} 
orthonormal basis as opposed to estimates for a favourable one that may exist due
to additional assumptions or structures. However, the latter one seems to be the question
much less studied in the literature, see \cite{sogge-z:ei} or \cite{tz:bounds}
for some
partial discussions.
\end{rem}

We now give an example of the above remark
in the case of the the sphere $\mathbb{S}^3\simeq \SU2$. We consider the Laplacian 
(the Casimir element) $E=-\mathcal{L}_{\mathbb{S}^3}$. 
We will apply the condition given by Theorem \ref{main12ab} 
along with the control $\Lambda(\ell, \infty)=d_{\ell}^{\frac 14}$. 
For the symbol of $(I+E)^{-\frac{\alpha}{2}}$, 
since the eigenvalues of $I+E$ are of the form $(1+\ell)\ell$  we obtain
\begin{align*} \|\sigma_{(I+E)^{-\frac{\alpha}{2}}}(\ell)\|_{S^r}^r=
((1+\ell)\ell)^{-\frac{\alpha r}{2}}d_{\ell}
\approx ((1+\ell)\ell)^{-\frac{\alpha r}{2}}\ell^2
\approx (1+\ell^2)^{1-\frac{\alpha r}{2}}.
\end{align*} 
Therefore, using $d_{\ell}\approx \ell^{2}$, 
\begin{align*} \sum\limits_{\ell}\|\sigma_{(I+E)^{-\frac{\alpha}{2}}}(\ell)\|_{S^r}^r\Lambda(\ell, \infty)^{(\tilde{p_2}+\tilde{q_1})r}\leq&\sum\limits_{\ell}(1+\ell ^2)^{1-\frac{\alpha r}{2}}\ell^{\half(\tilde{p_2}+\tilde{q_1})r}\\
\approx&\sum\limits_{\ell}(1+\ell)^{2-{\alpha r}+\half(\tilde{p_2}+\tilde{q_1})r}.
\end{align*}
The series on the right-hand side converges if and only if $2-{\alpha r}+\half(\tilde{p_2}+\tilde{q_1})r<-1$. Thus,  
the condition \[\alpha>\frac{3}{r}+\half(\tilde{p_2}+\tilde{q_1})\]
ensures the membership of $(I+E)^{-\frac{\alpha}{2}}$ in the Schatten class of order $r$. Summarising, we have proved the following:

\begin{cor}\label{COR:su2-LapM}
If $\alpha > \frac 3r+\half(\tilde{p_2}+\tilde{q_1})$, $0<r\leq 1$, the operator
$(I-\mathcal{L}_{\mathbb{S}^3})^{-\frac{\alpha}{2}}$ is $r$-nuclear from 
$L^{p_1}(\mathbb{S}^3)$ into $L^{p_2}(\mathbb{S}^3)$.
\end{cor}
Corollary \ref{COR:su2-LapM} gives a direct proof of Corollary 3.19 in
\cite{dr13a:nuclp} which was proved there in the group setting.

\begin{rem} It is clear that the sharpness of the sufficient conditions obtained in this section depends on how sharp is the $\Lambda$ function we can choose. 
 For instance the best situation for $\Lambda(\ell,\infty)$ is when it can be chosen constant,
  i.e. when the eigenfunctions are uniformly bounded. This is the case of the torus $\tn$ and unfortunately may be essentially the only one, see \cite{tz:bounds}.
\end{rem}


\end{document}